\documentclass{article}
\usepackage{amsmath,amssymb,enumerate,bbm,color,theorem}
\usepackage[american]{babel}
\usepackage{tikz}

\def\ds{\displaystyle}

\newcommand{\emptytree}{\begin{tikzpicture}[scale=0.7,baseline=-2pt]
  \node (r) at (0,0) {$\emptyset$};
 \end{tikzpicture}}
 \newcommand{\eun}{\begin{tikzpicture}[scale=0.7,baseline=-2pt]
  \node (r) at (0,0) {${\eta}_1$};
 \end{tikzpicture}}
  \newcommand{\edeux}{\begin{tikzpicture}[scale=0.7,baseline=-2pt]
  \node (r) at (0,0) {${\eta}_2$};
 \end{tikzpicture}}
 \newcommand{\etrois}{\begin{tikzpicture}[scale=0.7,baseline=-2pt]
  \node (r) at (0,0) {${\eta}_3$};
 \end{tikzpicture}}
 \newcommand{\equatre}{\begin{tikzpicture}[scale=0.7,baseline=-2pt]
  \node (r) at (0,0) {${\eta}_4$};
 \end{tikzpicture}}
 \newcommand{\edeuxetrois}{\begin{tikzpicture}[scale=0.7,baseline=-2pt]
  \node (r) at (0,0) {${\eta}_2$};
  \node (a) at (0.6,0) {${\eta}_3$};
 \end{tikzpicture}}
 
 \newcommand{\Szeroeun}{\begin{tikzpicture}[scale=0.7,baseline=5pt]
  \node (r) at (0,0) {$0$};
  \node (a) at (0,1) {${\eta}_1$};
    \draw (r) --(a) ;
 \end{tikzpicture}}
 \newcommand{\Seunedeux}{\begin{tikzpicture}[scale=0.7,baseline=5pt]
  \node (r) at (0,0) {$\eta_1$};
  \node (a) at (0,1) {${\eta}_2$};
    \draw (r) --(a) ;
 \end{tikzpicture}}
 \newcommand{\Seunetrois}{\begin{tikzpicture}[scale=0.7,baseline=5pt]
  \node (r) at (0,0) {$\eta_1$};
  \node (a) at (0,1) {${\eta}_3$};
    \draw (r) --(a) ;
 \end{tikzpicture}}
  \newcommand{\Setroisequatre}{\begin{tikzpicture}[scale=0.7,baseline=5pt]
  \node (r) at (0,0) {${\eta}_3$};
  \node (a) at (0,1) {${\eta}_4$};
    \draw (r) --(a) ;
 \end{tikzpicture}}
  \newcommand{\Secinqesix}{\begin{tikzpicture}[scale=0.7,baseline=5pt]
  \node (r) at (0,0) {${\eta}_5$};
  \node (a) at (0,1) {${\eta}_6$};
    \draw (r) --(a) ;
 \end{tikzpicture}}
 \newcommand{\Czeroedeuxetrois}{
 \begin{tikzpicture}[scale=0.7,baseline=5pt]
  \node (r) at (0,0) {$0$};
  \node (a) at (-0.5,1) {${\eta}_2$};
  \node (b) at (0.5,1) {${\eta}_3$};
  \draw (r) --(a) ;
  \draw (r) --(b) ;
\end{tikzpicture}}
 \newcommand{\Ceunedeuxetrois}{
 \begin{tikzpicture}[scale=0.7,baseline=5pt]
  \node (r) at (0,0) {${\eta}_1$};
  \node (a) at (-0.5,1) {${\eta}_2$};
  \node (b) at (0.5,1) {${\eta}_3$};
  \draw (r) --(a) ;
  \draw (r) --(b) ;
\end{tikzpicture}}
 \newcommand{\Czeroeunedeuxetrois}{\begin{tikzpicture}[scale=0.7,baseline=5pt]
  \node (r) at (0,0) {$0$};
  \node (a) at (-0.5,1) {${\eta}_1$};
  \node (b) at (0,1) {${\eta}_2$};
   \node (c) at (0.5,1) {${\eta}_3$};
  \draw (r) --(a) ;
  \draw (r) --(b) ;
  \draw (r) --(c) ;
\end{tikzpicture}}
 \newcommand{\BeunedeuxSetroisequatre}{\begin{tikzpicture}[scale=0.7,baseline=5pt]
  \node (r) at (0,0) {${\eta}_1$};
  \node (a) at (-0.5,1) {${\eta}_2$};
  \node (b) at (0.5,1) {${\eta}_3$};
  \node (c) at (0.5,2) {${\eta}_4$};
    \draw (r) --(a) ;
     \draw (r) --(b) --(c);
 \end{tikzpicture}}
\newcommand{\BzeroetroisSedeuxeun}{\begin{tikzpicture}[scale=0.7,baseline=5pt]
  \node (r) at (0,0) {$0$};
 \node (a) at (-0.5,1) {${\eta}_3$};
  \node (b) at (0.5,1) {${\eta}_2$};
  \node (c) at (0.5,2) {${\eta}_1$};
    \draw (r) --(a) ;
     \draw (r) --(b) --(c);
\end{tikzpicture}} 
\newcommand{\BzeroedeuxSetroiseun}{\begin{tikzpicture}[scale=0.7,baseline=5pt]
  \node (r) at (0,0) {$0$};
 \node (a) at (-0.5,1) {${\eta}_2$};
  \node (b) at (0.5,1) {${\eta}_3$};
  \node (c) at (0.5,2) {${\eta}_1$};
    \draw (r) --(a) ;
     \draw (r) --(b) --(c);
\end{tikzpicture}} 
\newcommand{\BeBeunedeuxSetroisequatreSecinqesix}{\begin{tikzpicture}[scale=0.7,baseline=5pt]
  \node (r) at (0,0) {$\eta$};
 \node (a) at (-0.5,1) {${\eta}_1$};
 \node (b) at (-1,2) {${\eta}_2$};
  \node (c) at (0,2) {${\eta}_3$};
  \node (d) at (0,3) {${\eta}_4$};
   \node (e) at (0.5,1) {${\eta}_5$};
    \node (f) at (0.5,2) {${\eta}_6$};
    \draw (r) --(a)--(b);
     \draw (a) --(c)--(d);
     \draw (r) --(e) --(f);
\end{tikzpicture}} 

\newcommand{\Eta}{\mathrm{H}}
\newcommand{\emdash}{---}
\newcommand{\nocomma}{}
\newcommand{\tmem}[1]{{\em #1\/}}
\newcommand{\tmmathbf}[1]{\ensuremath{\boldsymbol{#1}}}
\newcommand{\tmop}[1]{\ensuremath{\operatorname{#1}}}
\newcommand{\tmrsub}[1]{\ensuremath{_{\textrm{#1}}}}
\newcommand{\tmscript}[1]{\text{\scriptsize{$#1$}}}
\newcommand{\tmstrong}[1]{\textbf{#1}}
\newcommand{\tmtextit}[1]{{\itshape{#1}}}
\newcommand{\tmtexttt}[1]{{\ttfamily{#1}}}
\newenvironment{enumeratenumeric}{\begin{enumerate}[1.] }{\end{enumerate}}

\newenvironment{itemizedot}{\begin{itemize} }{\end{itemize}}
\newenvironment{itemizeminus}{\begin{itemize} }{\end{itemize}}
\newenvironment{proof}{\noindent\textbf{Proof\ }}{\hspace*{\fill}$\Box$\medskip}
\definecolor{grey}{rgb}{0.75,0.75,0.75}
\definecolor{orange}{rgb}{1.0,0.5,0.5}
\definecolor{brown}{rgb}{0.5,0.25,0.0}
\definecolor{pink}{rgb}{1.0,0.5,0.5}
\newtheorem{definition}{Definition}
\newtheorem{lemma}{Lemma}
\newtheorem{proposition}{Proposition}
{\theorembodyfont{\rmfamily}\newtheorem{remark}{Remark}}
\newtheorem{theorem}{Theorem}

\begin{document}

\title{Ecalle's arborification--coarborification transforms and
Connes--Kreimer Hopf algebra\\  Les transformations d'arborification--coarborification d'Ecalle et l'alg\`ebre de Hopf de Connes--Kreimer}\author{\and Fr{\'e}d{\'e}ric Fauvet and
Fr{\'e}d{\'e}ric Menous}\maketitle

\begin{abstract}
  We give a natural and complete description of Ecalle's mould--comould
  formalism within a Hopf--algebraic framework. The arborification transform
  thus appears as a factorization of characters, involving the shuffle or
  quasishuffle Hopf algebras, thanks to a universal property satisfied by
  Connes--Kreimer Hopf algebra. We give a straightforward characterization of
  the fundamental process of homogeneous coarborification, using the explicit
  duality between decorated Connes--Kreimer and Grossman--Larson Hopf algebras.
  Finally, we introduce a new Hopf algebra that systematically underlies the
  calculations for the normalization of local dynamical systems.

\end{abstract}

\begin{abstract}

  Nous donnons une description compl{\`e}te et naturelle du formalisme
d'arborification/coarborification d'Ecalle en termes d'alg{\`e}bres de Hopf.
L'arborification appara{\^i}t alors comme une factorisation de caract{\`e}res,
impliquant les alg{\`e}bres shuffle ou quasishuffle, en vertu d'une
prori{\'e}t{\'e} universelle satisfaite par l'alg{\`e}bre de Connes--Kreimer.
Dans ce cadre, nous obtenons de fa{\c c}on directe le proc{\'e}d{\'e}
fondamental de coarborification homog{\`e}ne, en utilisant la dualit{\'e}
explicite entre les alg{\`e}bres de Hopf d{\'e}cor{\'e}es de Connes--Kreimer et
Grossman--Larson. Enfin, nous introduisons une nouvelle alg{\`e}bre de Hopf
qui est sous--jacente aux calculs de normalisation des syst{\`e}mes dynamiques locaux.

\end{abstract}


Fa\`a di Bruno, Moulds, Arborification, Coarborification

\vspace{0.3cm}


{\tableofcontents}

{\newpage}


\section{Introduction}\label{s:intro}

The local study of dynamical systems, through normalizing transformations,
involves calculations in groups, or pseudogroups, of diffeomorphisms (e. g. formal, or
analytic) that are tangent to identity. Other situations where these
explicit calculations are required are also numerous in key questions of
classification of singular geometric structures. Another source of examples is
given by the so called mechanism of Birkhoff decomposition
({\cite{menous_birkh}}). The group $\tmmathbf{G}$ of formal tangent to
Identity diffeomorphisms is the one in which most of the calculations are to
be performed.

To tackle problems of this kind, Jean Ecalle has developped a powerful
combinatorial environment, named {\tmem{mould calculus}}, that leads to
formulas that are surprisingly explicit. This calculus has lately been the
object of attention within the algebraic combinatorics community
({\cite{chapoton}}, {\cite{thibon_et_al}}). However, despite its striking
achievements, this formalism has been little used in local dynamics, in
pending problems that anyway seem out of reach of other approaches. One reason
might be that it uses a sophisticated system of notations, in which a number
of infinite sums are manipulated, in a way that calls for a number of proofs
and explanations, that are to a certain extent still missing in the few
existing papers using mould calculus. Beside this, some constructions
introduced by Ecalle, though obviously appearing as extraordinarily efficient,
might remain a bit mysterious; an example of this is the so called
{\tmem{homogeneous coarborification}} \ ({\cite{snag}}), for which we are now
able to give in the present paper a very natural algebraic presentation.

In fact, Ecalle's mould--comould formalism can be very naturally recast in a
Hopf--algebraic setting, with the help of a number of Hopf algebras (shuffle,
quasi--shuffle, their graded duals, etc) which are now widely used within
algebraic combinatorics. In the present text, we show how this can be done,
which makes it possible to give simple and quick proofs of important
properties regarding mould calculus.

As is now well known ({\cite{figuer}}, {\cite{car_kef_fig}}), the
Hopf--algebraic formulation of computations on formal diffeomorphisms involves
the so called Fa{\`a} di Bruno Hopf algebra, which encodes the eponymous
formula for higher order chainrule. In fact, Hopf algebraic tools and concepts
have very recently become pervasive in dynamical systems, see e.g. \cite{lund-mk} and the
references therein. Now, {\tmem{an essential point}} is the following: the
reformulation of a classification problem through the use of Fa{\`a} di Bruno
Hopf algebra (or, more simply, calculations on compositions of diffeomorphisms
involving the Fa{\`a} di Bruno formula), although satisfactory at the formal
level \ will usually be inefficient, in the hard cases, for the question of
analyticity of the series.{\tmem{ Indeed, in difficult situations involving
resonances and/or small denominators, the formulas obtained through Fa{\`a} di
Bruno are most of the time not explicit enough to obtain satisfactory growth
estimates on the coefficients}}.

On the other hand, Ecalle's mould--comould expansions often lead to explicit
coefficients but, when trying to control the size of these in a
straightforward way, we often encounter systematic divergence, which claims
for the introduction of something subtler.

So the need was for some sort of {\tmstrong{{\tmem{intermediate Hopf
algebra}}}}, in which the algebraic calculations would still be tractable, and
leading to explicit formulas from which key estimates can be obtained, to
eventually get e.g. the analyticity properties we could expect. This is
exactly what arborification/coarborification does. Once the original
definitions of Ecalle are translated into a Hopf--algebraic setting, with the
use of Connes-Kreimer Hopf algebra $\tmop{CK}$ and its graded dual, it is
possible to recognize that the arborification transform is nothing else that a
property of factorisation of characters between Hopf algebras (we perform this
at the same time for the shuffle and quasishuffle cases), using the fact that
$\tmop{CK}$ is an initial object for Hochschild cohomology for a particular
category of cogebras ({\cite{ck}}, {\cite{foissy1}}, {\cite{foissy2}}).

Thus, the universality of the arborification mechanism is directly and
naturally connected with a universal property satisfied by Connes-Kreimer Hopf
algebra, whose importance is by now widely acknowledged (see e. g.
{\cite{figuer}}).

The paper is organized as follows. In the next section, we recall a few basic
facts \ concerning normalization in local dynamics, focusing on two basic
situations for which it is possible to introduce all the relevant objects in a
simple, yet non trivial, context. The following section is devoted to an
algebraic study of the group of tangent to identity formal diffeomorphisms,
introducing at this stage the Fa{\`a} di Bruno Hopf algebra
$\mathcal{H}_{\tmop{FdB}}$. This section doesn't contain new results, yet we
have chosen a presentation stressing the role of substitution automorphisms,
and adopting a systematic way of looking at normalizing equations as equations
on characters of Hopf algebras which are by now classical objects (basic
terminology and facts on graded Hopf algebras are included).

Then we are ready to interpret moulds, at least the ones with symmetry
properties that are met in practice, as characters or infinitesimal characters
on some classical Hopf algebras, namely symmetral (resp. symmetrel) moulds as
characters of the shuffle (resp. quasishuffle) Hopf algebra. This is the
object of section 4, where the basic notions regarding moulds, comoulds and
their ``contractions'' are given.

In section \ref{s:arbor} the key dual mechanisms of arborification and
coarborifications are introduced, and described through the introduction of
$\tmop{CK}$ and its graded dual, known to be isomorphic to the
Grossman--Larson Hopf algebra

In fact, we show that the natural isomorphism between these two Hopf algebras
leads directly, in the contexts of comoulds, to the process of
{\tmem{homogeneous coarborification}}, which was put forward by Ecalle with
very little explanation. A cautious handling of the symmetry factors of the
trees is crucial, here.

In section \ref{s:dioph} we describe the Hopf algebra $\tmop{CK}^+$ which is ultimately used
in practical calculations of normalizing transformations, for questions of
classification of dynamical systems, involving resonances and small
denominators. This solves at the same time an algebraic problem and a
essential analytic one, regarding the growth estimates of the coefficients of
the diffeomorphisms. The point of view which is enhanced in the present paper
can be summed up in the following considerations:
\begin{itemizedot}
  \item The systematic of substitution automorphisms, which constitute an
  alternative -- a very profitable one, because it is more {\tmem{flexible}}
  {\emdash} to changes of variables, naturally entail a Hopf--algebraic
  presentation
  
  \item Calculations in the Fa{\`a} di Bruno Hopf algebra are a direct mirror of the
  traditional approach through normalizing transformations, yet they don't
  yield results which are explicit enough to tackle difficult cases
  
  \item There is a {\tmem{hierarchy}} Sh/Qsh, $\tmop{CK}$, $\tmop{CK}^+$ of
  Hopf algebras, the first ones adapted to the simple formal classification
  results, the second one necessary for controlling the regularity of the
  formal constructions, under a strong non resonance condition, and the last
  one to take care of objects satisfying a weak nonresonance condition
\end{itemizedot}
The main results of the text are thus the ones which concern the Hopf algebra
$\tmop{CK}^+$, which is the fundamental one to be used by the practitioner, in
difficult problems involving small denominators.

The research leading these results was partially supported by the French National Research Agency under the reference ANR-12-BS01-0017.


\section{Normal forms}\label{s:normforms}

To study a dynamical system, a standard procedure, since Poincar{\'e}, is to
try to conjugate the object to another one which is as simple as possible and
for which the dynamics is well understood, and which is then called a normal
form. The classes of objects that are considered in the present text are
vector fields and diffeomorphisms in $\mathbbm{C}^{\nu}$ and we study them
near a singularity, namely a vanishing point for the field or a \ fixed point
for the diffeomorphism.

Conjugation is thus obtained through the action of a change of coordinates,
performed by a diffeomorphism leaving the singular point invariant. When we
consider, say, an analytic germ of vector field $X$ at the origin, the
simplest field to which we can hope to conjugate it through an analytic change
of coordinates is the linear part $X^{\tmop{lin}}$ of $X$. When trying to do
this, one immediately encounters the possibility of obstructions; indeed, if
the eigenvalues of $X^{\tmop{lin}}$ (supposed semi--simple) are $\lambda_1,
\ldots, \lambda_{\nu}$ (with possible multiplicities), then even the
{\tmem{formal}} conjugation of $X$ to $X^{\tmop{lin}}$ is not possible when
some combinations of the following type do exist:

\begin{equation}
  m_1 \lambda_1 + \ldots m_{\nu} \lambda_{\nu} - \lambda_i = 0 \label{e:reson}
\end{equation}
in this relation, $i \in \{ 1, \ldots, \nu \} ; \tmop{the} m_j \tmop{are}
\tmop{nonnegative} \tmop{integers}, \tmop{with} \sum m_j \geqslant 2$.

Such a relation can also be written as $\langle n, \lambda \rangle = 0$ where
$\lambda = (\lambda_1, \ldots, \lambda_{\nu})$ is the spectrum, $n$ is a
$\nu$--uple of integers that belongs to the following set:

$$\mathcal{N}= \left\{ (n_1, \ldots, n_{\nu}) ; n_i \geqslant - 1, \tmop{at}
\tmop{most} \tmop{one} \tmop{being} = - 1, \tmop{and} \sum_1^{\nu} n_i
\geqslant 1 \right\}$$

In the sequel, we shall use the following notation: $\langle n, \lambda
\rangle = \Sigma n_i \lambda_i$. A relation such as (\ref{e:reson}) is called a {\tmem{resonance}}, and we
focus now on the nonresonant case, for which by definition no resonance
exists. A standard way to obtain the linearization is then to conjugate with
polynomial changes of coordinates our given field to its linear part, up to
terms of a given valuation and then composing these transforms in order to
obtain in the end only the linear part. The absence of resonance ensures that
each step is possible, and the formal convergence of the infinite product is
easy. In the next section, an alternative method is described, which directly
lead to moulds. In this way we obtain a unique linearizing transformation, if
we impose that it be tangent to identity.

Technically, it might happen that under the non resonance condition given
above, some of the partial sums might vanish. If we want to avoid this, we
have to consider a stronger non resonance condition, namely that the
$\lambda_i$ are independant over $\mathbbm{Z}$. We shall below work out the
algebraic formulation using the strong condition, and eventually in section 6
we will be able to cope with the weaker one, once the appropriate Hopf algebra
has been defined.

However, the formal transform will not always be convergent : in the process
of computation of the linearization transform, whatever the chosen method, we
encounter divisions by expressions $m_1 \lambda_1 + \ldots m_{\nu}
\lambda_{\nu} - \lambda_i$, which, although non zero, might be very small.
This problem of occurence of {\tmem{small denominators}} calls for an extra
hypothesis on the spectrum, in order to control the size of the coefficients
of the series we are interested in. The original breakthrough was made by
Siegel, and the best (it is known to be optimal in dimension 2) diophantine
condition so far, is Brjuno's condition:
\[ \sum \frac{1}{2^k} \tmop{Log} ( \frac{1}{\Omega (2^{k + 1})}) < \infty \]
where $\Omega (h) = \min \left\{ | \langle n, \lambda \rangle |, \sum n_i
\leqslant h \right\}$

Under this condition it can be shown (Brjuno, \cite{brjuno}) that the
normalization tranform is indeed \ analytic.

The classification problem for germs of diffeomorphisms goes along the same
lines: we consider a diffeomorphism $\varphi$ at the origin of $0$ in
$\mathbbm{C}^{\nu}$ and we wish to conjugate it to its linear part
$\varphi^{\tmop{lin}}$, with $\varphi^{\tmop{lin}}  (x) = (l_1 x_1, \ldots,
l_{\nu} x_{\nu})$. In that case, a resonance can be written as
\[ l_1^{m_1} \ldots l_{\nu}^{m_{\nu}} - 1 = 0 \]
with the exponents $( m_i) $as above: \ $m = (m_1, \ldots, m_{\nu}) \tmop{is}
\tmop{in} \mathcal{N}$. In the absence of resonance, such a diffeomorphism is formally
conjugate to its linear part, and this can be proved by the same method as for
fields.

Here also, we shall have to consider a strong non resonance condition, namely
that no relation of the above type vanishes, for any family of coefficients
$m_i$ in $\mathbbm{Z}$.

Under the following diophantine hypothesis, it is known ({\cite{russman}})
that the linearizing transform is analytic.
\[ \sum \frac{1}{2^k} \tmop{Log} ( \frac{1}{\Omega (2^{k + 1})}) < \infty \]
where $\Omega (h) = \min \left\{  | l_1^{m_1} \ldots l_{\nu}^{m_{\nu}} - 1 |,
\sum m_i \leqslant h \right\}$

In dimension one, there is a unique tangent to identity formal diffeomorphism $h$ that
conjugates a given diffeomorphism $g : x \longrightarrow \lambda x + \Sigma g_n x^n$ \
to its linear part $g_l$ , provided $g_l$ is not a periodic rotation (the non
resonant case), and its coefficients are given by an explicit but already
somewhat complicated recursive expression :
\[ h_n = \frac{1}{\lambda^n - \lambda} \left[ g_n + \sum_2^{n - 1} g_p
   \sum_{j_1 + \ldots + j_p = n_{}} h_{j_1} \ldots h_{j_p} \right] \]

These formulas are of little help in directly proving the most delicate
analytic linearization results, already in the lowest dimension, let alone in
dimension greater than 1.

As a remark, let us mention that the required calculations involving
compositions of diffeomorphisms are essentially of the same type when one is
interested in classifications of geometric structures with singularities.
 Consequently, the algebraic formalism developped below
can also be used for these problems, in cases where the complexity of the
problem tends to make other techniques inoperant (e.g. singular Poisson
structures displaying resonances, in a context of small denominators).

Although we consider as examples the cases of non--resonant germs of vector
fields or germs of diffeomorphisms, in any dimension at the origin of
$\mathbbm{C}$, all the algebraic structures, as well as Ecalle's constructions
that come into play by following these basic situations as leading thread, \
are of a universal nature, as notably the Hopf algebra of section 6.

\section{Algebraic structures on the group $\tmmathbf{G}$ of tangent to identity
diffeomorphisms}\label{s:G}

\subsection{The Lie algebra $\tmmathbf{g}$ of formal vector fields.}

We consider now the group $\tmmathbf{G}$ of formal diffeomorphisms that are
tangent to Identity at the origin of $\mathbbm{C}^\nu$:
\[
\tmmathbf{G} = \{ \varphi = ({\varphi}_1 , ... , {\varphi}_\nu ): x=(x_1,\dots,x_{\nu})\longmapsto x+ \text{h.o.t.} \}
\]
It is well-known that the group ${\tmmathbf{G}}$ is the
Lie group of the Lie algebra $\tmmathbf{g}$ of formal vector fields:
\[ \tmmathbf{g}= \left\{ X = \sum_{i = 1}^{\nu} X_i (x) \partial_{x_i},
   \hspace{1em} X_i (x) \in \mathbbm{C}_{\geqslant 2} [[x]] \right\} \]
where $\mathbbm{C}_{\geqslant 2} [[x]]$ denotes formal power series in the variables $x=(x_1,\dots,x_{\nu})$ of total
valuation greater than 1. Even if we deal with formal power series, the
``geometric'' interpretation goes as follows: for a given vector field $X$,
consider the differential system:
\[ \left\{ \begin{array}{lll}
     y_1' (t) & = & X_1 (y_1 (t), \ldots, y_{\nu} (t))\\
     & \vdots & \\
     y_{\nu}' (t) & = & X_{\nu} (y_1 (t), \ldots, y_{\nu} (t))
   \end{array} \right. \]
with the initial conditions $y (0) = (y_1 (0), \ldots, y_{\nu} (0)) = (x_1,
\ldots, x_{\nu}) = x$. Even formally, the solution at time $t$ is given by $y
(t) = \varphi^t (x)$ where $\varphi^t \in \tmmathbf{G}$ and $\varphi^t \circ
\varphi^s = \varphi^{t + s}$. Namely, $\varphi^t$ is the flow of the vector
field $X$, whose exponential is simply $\exp (X) = \varphi^1$.

This correspondance is bijective ($X = \log (\varphi)$) as we shall see in
the following section. Note that the computations are not so easy to handle
but become clear, once diffeomorphisms are interpreted through their action on
formal power series.

\subsection{The action of $\tmmathbf{G}$ and substitution automorphisms.}

From the definition of $\tmmathbf{g}$ it is easy to derive its action on a
formal power series $f$ , by the chain--rule formula: $( f ( x ( t)))' = (
X.f) ( x ( t))$. If $X = \sum_{i = 1}^{\nu} X_i (x) \partial_{x_i}$,
\[ X.f = \sum_{i = 1}^{\nu} X_i (x) \partial_{x_i} f \]
as a vector field is a differential operator. Moreover, it is a derivation on
$\mathbbm{C} [[x]]$ since
\[ X. (f g) = (X.f) g + f (X.g) \]
Similarly the natural action of a diffeomorphism $\varphi$ on a series $f$ is
given by
\[ (f \vartriangleleft \varphi) (x) = (\Theta_{\varphi} .f) (x) = f \circ \varphi
   (x) \]
This defines a right action of the group $\tmmathbf{G}$ on the algebra
$\mathbbm{C} [[x]]$ and $\Theta_{\varphi}$ is the substitution automorphism
associated to $\varphi$ :
\[ \begin{array}{c}
     \Theta_{\varphi} .\Theta_{\psi} .f = (f \vartriangleleft \psi) \vartriangleleft
     \varphi = \Theta_{\psi \circ \varphi} .f\\
     \Theta_{\varphi} . (f g) = (f g) \vartriangleleft \varphi = (f
     \vartriangleleft \varphi) (g \vartriangleleft \varphi) = (\Theta_{\varphi} .f)
     (\Theta_{\varphi} .g)
   \end{array} \]
Let us focus on such substitution automorphisms, since they are one of the key
ingredients to perform mould calculus.

\begin{proposition}
  Let $\tilde{\tmmathbf{G}}$ be the subset of linear endomorphisms $\Theta$,
  which are continuous wrt the Krull topology \ of $\mathbbm{C} [[x]]$ such
  that
  \begin{enumeratenumeric}
    \item $\Theta (x) = \Theta .x = \varphi (x) \in \tmmathbf{G}$.
    
    \item For any series $f$, $g$, $\Theta (f g) = \Theta . (f g) = (\Theta
    .f) (\Theta .g)$
  \end{enumeratenumeric}
  then $\tilde{\tmmathbf{G}}$ is a group (the group of substitution
  automorphisms) and the (``evaluation'') map $\tmop{ev}$ defined by
  \[ \tmop{ev} (\Theta) = \Theta .x \in \tmmathbf{G} \]
  is an anti--isomorphism of groups.
\end{proposition}

\begin{proof}
  The proof is straightforward : consider a monomial $x^n = x_1^{n_1} \ldots
  x_{\nu}^{n_{\nu}}$ . Because of the second property,
  \[ \Theta . (x^n) = \varphi_1^{n_1} \ldots \varphi_{\nu}^{n_{\nu}}
     \hspace{1em} (\Theta .x = \varphi (x) = (\varphi_1 (x), \ldots,
     \varphi_{\nu} (x)) \]
  By linearity and continuity,
  \[ \Theta .f = f \circ \tmop{ev} (\Theta) \]
  and the proposition follows, noticing that $ev (\Theta_1 \Theta_2) = \varphi_2 \circ \varphi_1$ (whence the anti--isomorphism property).
\end{proof}

In the sequel we shall identify $\tmmathbf{G}$ with $\tilde{\tmmathbf{G}}$
when needed, taking advantage of the fact that such substitution
automorphisms, as vector fields in $\tmmathbf{g}$, can be seen as differential
operators :

\begin{proposition}
  Let $\varphi = x + u (x) \in \tmmathbf{G}$, then
  \[ \tmop{ev}^{- 1} (\varphi) = \Theta_{\varphi} = \tmop{Id}_{\mathbbm{C}
     [[x]]} + \hspace{-3mm}\sum_{\tmscript{\begin{array}{c}
       n_1 + \ldots + n_{\nu} \geqslant 1\\
       n_i \geqslant 0
     \end{array}}} \hspace{-2mm}\frac{1}{n_1 ! \ldots n_{\nu} !} u_1^{n_1} (x) \ldots
     u_{\nu}^{n_{\nu}} (x) \partial_{x_1}^{n_1} \ldots
     \partial_{x_{\nu}}^{n_{\nu}} \]
\end{proposition}

This is simply the Taylor formula :
\[ \begin{array}{ccc}
     \Theta_{\varphi} .f (x) & = & f (x + u (x))\\
     & = & \ds f (x) + \hspace{-5mm}\sum_{\tmscript{\begin{array}{c}
       n_1 + \ldots + n_{\nu} \geqslant 1\\
       n_i \geqslant 0
     \end{array}}} \hspace{-2mm}\frac{1}{n_1 ! \ldots n_{\nu} !} u_1^{n_1} (x) \ldots
     u_{\nu}^{n_{\nu}} (x) \partial_{x_1}^{n_1} \ldots
     \partial_{x_{\nu}}^{n_{\nu}} f (x)
   \end{array} \]

There is still some work to do to perform mould calculus, but one can already
use this to define explicitely the exponential of a vector field. If $X \in
\tmmathbf{g}$, then, for any real number $t$, the differential operator
\[ \Theta^t = \exp (t X) = \tmop{Id} + \sum_{s \geqslant 1} \frac{t^s}{s!} X^s
\]
is a well-defined substitution automorphism and, if $\varphi^t = \tmop{ev}
(\Theta^t)$, it is the flow of the vector field $X$. Conversely, for a given
diffeomorphism $\varphi$, if $\Theta$ is its substitution automorphism then
this is the flow at time $t = 1$ of the vector field
\[ X = \log (\Theta) = \log (\tmop{Id} + (\Theta - \tmop{Id})) = \sum_{s
   \geqslant 1} \frac{(- 1)^{s - 1}}{s} (\Theta - \tmop{Id})^s \]
We leave the proof to the reader (see also {\cite{il-yak}}).

\subsection{Degrees and homogeneous components.}

Vector fields and diffeomorphisms are made of power series, namely series of
monomials $x^n = x_1^{n_1} \ldots x_{\nu}^{n_{\nu}}$ of degree $n = (n_1, \ldots,
n_{\nu})$, but what is relevant for such objects in the context of
normalization is not the degree, but the notion of {\tmem{homogeneous
components}} related to their action on monomials.

More precisely, a formal power series is given by
\[ f (x) = \sum_{n \in \mathbbm{N}^{\nu}} f_n x^n \]
where $n = (n_1, \ldots, n_{\nu}) \in \mathbbm{N}^{\nu}$ is the degree of $x^n
= x_1^{n_1} \ldots x_{\nu}^{n_{\nu}}$ and $|n| = n_1 + \ldots + n_{\nu}$ is
its total degree. Such monomial are very well adapted to the product of power
series: if \ $f (x) = \ds \sum_{n \in \mathbbm{N}^{\nu}} f_n x^n$ and $g (x) =
\ds \sum_{n \in \mathbbm{N}^{\nu}} g_n x^n$, then their product $h (x) = f (x) g
(x) = \ds \sum_{n \in \mathbbm{N}^{\nu}} h_n x^n$ is such that
\[ h_n = \sum_{k + l = n} f_k g_l \]
but, for example, if one considers an elementary vector field $X_{i, n} = x^n
\partial_{x_i}$ then
\[ X_{i, n} .x^m = m_i x^{n + m - e_i} \]
where $e_i$ is the element of $\mathbbm{N}^{\nu}$ whose $i^{\tmop{th}}$ entry
(resp. $j^{\tmop{th}}$ entry with $j \neq i$) is 1 (resp. 0). Regarding its
action on monomials such a vector field is ``homogeneous'' of degree $\eta = n
- e_i$ and this will be the right notion of degree for vector fields and
diffeomorphisms. This suggests the following notation : for $1 \leqslant i
\leqslant \nu$, let
\[ H_i = \{\eta = n - e_i, \hspace{1em} n \in \mathbbm{N}^{\nu}, |n|
   \geqslant 2\}. \]
For any $1 \leqslant i \leqslant \nu$ and $\eta \in H_i$, $| \eta |
\geqslant 1$ and any vector field $X$ in $\tmmathbf{g}$ can be written
\[ X = \sum_{i = 1}^{\nu} \sum_{\eta \in H_i} b^i_{\eta} x^{\eta} x_i
   \partial_{x_i} = \sum_{\eta \in H} \sum_{i ; \eta \in H_i} b^i_{\eta}
   x^{\eta} x_i \partial_{x_i} \]
where $H = H_1 \cup \ldots \cup H_{\nu}$. We will note
\[ \mathbbm{B}_{\eta} = \sum_{i ; \eta \in H_i} b^i_{\eta} x^{\eta} x_i
   \partial_{x_i} = \sum_i b^i_{\eta} x^{\eta} x_i \partial_{x_i} \]
assuming that the sum is restricted to the indices $i$ such that $\eta \in
H_i$. With this notation, $X$ is decomposed in ``homogeneous''
 components:
\[ X = \sum_{\eta \in H} \mathbbm{B}_{\eta} \]
where, for any $n \in \mathbbm{N}^{\nu}$ and $\eta \in H$, $\mathbbm{B}_{\eta}
.x^n$ is a monomial of degree $n + \eta$ (resp. 0) if $n + \eta \in
\mathbbm{N}^{\nu}$ (resp. $n + \eta$ is not in $\mathbbm{N}^{\nu}$).

On the same way, a diffeomorphism $\varphi = (\varphi_1, \ldots,
\varphi_{\nu})$ is given by $\nu$ series
\[ \varphi_i (x) = x_i \left( 1 + \sum_{\eta \in H_i} \varphi^i_{\eta}
   x^{\eta} \right) \]
and if $\bar{H}$ is the additive semigroup generated by $H$, then its
associated substitution automorphism can be decomposed in homogeneous
components
\[ \Theta = \tmop{Id} + \sum_{\eta \in \bar{H}} \mathbbm{D}_{\eta} \]
where
\[ \mathbbm{D}_{\eta} = \sum_{s \geqslant 1} \sum_{\tmscript{\begin{array}{c}
     1 \leqslant i_1, \ldots, i_s \leqslant \nu\\
     \eta = \eta_1 + \ldots + \eta_s\\
     \eta_k \in H_{i_k}
   \end{array}}} \frac{1}{s!} \varphi^{i_1}_{\eta_1} \ldots
   \varphi^{i_s}_{\eta_s} x^{\eta + e_{i_1} + \ldots + e_{i_s}}
   \partial_{x_{i_1}} \ldots \partial_{x_{i_s}} \]
with finite sums (for any given $\eta$), thanks to the fact that $| \eta_k |
\geqslant 1$ (thus $s \leqslant | \eta |$). This can be seen using the Taylor
expansion of $f \circ \varphi (x)$ and it gives a first flavour of mould
calculus :
\begin{itemizeminus}
  \item If $F = \exp (X)$ with $X = \sum_{\eta \in H} \mathbbm{B}_{\eta}$,
  then
  \[ \Theta = \tmop{Id} + \sum_{s \geqslant 1} \sum_{\eta_1, \ldots, \eta_s
     \in H} \frac{1}{s!} \mathbbm{B}_{\eta_s} \cdots \mathbbm{B}_{\eta_1} \]
  \item If $X = \log (\varphi)$ with $\Theta_{\varphi} = \tmop{Id} +
  \sum_{\eta \in \bar{H}} \mathbbm{D}_{\eta}$, then
  \[ X = \sum_{s \geqslant 1} \sum_{\eta_1, \ldots, \eta_s \in \bar{H}}
     \frac{(- 1)^{s - 1}}{s} \mathbbm{D}_{\eta_s} \cdots \mathbbm{D}_{\eta_1}
  \]
\end{itemizeminus}
These are in fact two examples of mould--comould expansions. We postpone now
the definition and study of mould expansions that will be very useful to deal
with linearization equations, namely when we conjugate a given dynamical
system to its linear part. But such decompositions can also be used to get the
Fa{\`a} di Bruno formulas.

\begin{proposition}
  Let $\varphi$ and $\psi$ in $\tmmathbf{G}$ and $\phi = \varphi \circ \psi$,
  then for $1 \leqslant i \leqslant \nu$,
  \[ \phi_i (x) = x_i \left( 1 + \sum_{\eta \in H_i} \phi^i_{\eta} x^{\eta}
     \right) \]
  with, for $\eta \in H_i$,
  \[ \phi^i_{\eta} = \varphi^i_{\eta} + \psi^i_{\eta} + \sum_{s \geqslant 2}
     \sum_{\tmscript{\begin{array}{c}
       i_1 = i\\
       1 \leqslant i_2, \ldots, i_s \leqslant \nu\\
       \eta = \eta_1 + \ldots + \eta_s\\
       \eta_k \in H_{i_k}
     \end{array}}} \frac{1}{(s - 1) !} P^{\eta_1 + e_i}_{i_2, \ldots, i_s}
     \hspace{1em} \varphi^{i_{}}_{\eta_1} \psi^{i_2}_{\eta_2} \ldots
     \psi^{i_s}_{\eta_s} \label{FdBf} \]
  where $P^{\eta_1 + e_i}_{i_2, \ldots, i_s}$ are integers, independant of
  $\varphi$ and $\psi$.
\end{proposition}

\begin{proof}
  Let $\varphi$ and $\psi$ in $\tmmathbf{G}$ and $\phi = \varphi \circ \psi$.
  We can write
  \[ \begin{array}{c}
       \Theta_{\varphi} = \ds \tmop{Id} + \sum_{\eta \in \bar{H}}
       \mathbbm{D}_{\eta}\\
       \Theta_{\psi} = \ds \tmop{Id} + \sum_{\eta \in \bar{H}}
       \mathbbm{E}_{\eta}
     \end{array} \]
  where
  \[ \begin{array}{c}\ds
       \mathbbm{D}_{\eta} = \sum_{s \geqslant 1}
       \sum_{\tmscript{\begin{array}{c}
         1 \leqslant i_1, \ldots, i_s \leqslant \nu\\
         \eta = \eta_1 + \ldots + \eta_s\\
         \eta_k \in H_{i_k}
       \end{array}}} \frac{1}{s!} \varphi^{i_1}_{\eta_1} \ldots
       \varphi^{i_s}_{\eta_s} x^{\eta + e_{i_1} + \ldots + e_{i_s}}
       \partial_{x_{i_1}} \ldots \partial_{x_{i_s}}\\
       \ds\mathbbm{E}_{\eta} = \sum_{s \geqslant 1}
       \sum_{\tmscript{\begin{array}{c}
         1 \leqslant i_1, \ldots, i_s \leqslant \nu\\
         \eta = \eta_1 + \ldots + \eta_s\\
         \eta_k \in H_{i_k}
       \end{array}}} \frac{1}{s!} \psi^{i_1}_{\eta_1} \ldots
       \psi^{i_s}_{\eta_s} x^{\eta + e_{i_1} + \ldots + e_{i_s}}
       \partial_{x_{i_1}} \ldots \partial_{x_{i_s}}
     \end{array} \]
  and $\phi_i (x) = \Theta_{\phi} .x_i = \Theta_{\varphi \circ \psi} .x_i =
  \Theta_{\psi} . \Theta_{\varphi} .x_i$. We get
  \[ \Theta_{\psi} . \Theta_{\varphi} = \tmop{Id} + \sum_{\eta \in
     \bar{H}} \mathbbm{D}_{\eta} + \sum_{\eta \in \bar{H}}
     \mathbbm{E}_{\eta} + \sum_{\eta_{}, \mu \in \bar{H}}
     \mathbbm{E}_{\mu} \mathbbm{D}_{\eta_{}} \]
  Now, $\tmop{Id} .x_i = x_i$, $\mathbbm{D}_{\eta} .x_i = \varphi^i_{\eta}
  x^{\eta} x_i$ (resp. 0) if $\eta \in H_i$ (resp. $\eta \not\in H_i$) and
  $\mathbbm{E}_{\eta} .x_i = \psi^i_{\eta} x^{\eta} x_i$ (resp. 0) if $\eta
  \in H_i$ (resp. $\eta \not\in H_i$) so it remains to compute
  $\mathbbm{E}_{\mu} \mathbbm{D}_{\eta_{}} .x_i$. This is zero as soon as
  $\eta_{} \not\in H_i$ and otherwise :
  \[ \begin{array}{ccc}
       \mathbbm{E}_{\mu} \mathbbm{D}_{\eta_{}} .x_i & = & \mathbbm{E}_{\mu} .
       (\varphi^i_{\eta_{}} x^{\eta_{} + e_i})\\
       & = & \varphi^i_{\eta_{}} \mathbbm{E}_{\mu} .x^{\eta_{} + e_i}\\
       & = & \varphi^i_{\eta_{}}\ds \left( \sum_{s \geqslant 1}
       \sum_{\tmscript{\begin{array}{c}
         1 \leqslant i_1, \ldots, i_s \leqslant \nu\\
         \mu = \eta_1 + \ldots + \eta_s\\
         \eta_k \in H_{i_k}
       \end{array}}} \frac{1}{s!} \psi^{i_1}_{\eta_1} \ldots
       \psi^{i_s}_{\eta_s} P^{\eta + e_i}_{i_1, \ldots, i_s} \right)
       x^{\eta_{} + \mu + e_i}
     \end{array} \]
  Where
  \[ P^{\eta + e_i}_{i_1, \ldots, i_s} = x^{- \eta - e_i + e_{i_1} + \ldots +
     e_{i_s}} \partial_{x_{i_1}} \ldots \partial_{x_{i_s}} x^{\eta_{} + e_i}
     \in \mathbbm{N} \]
  This gives the announced formula.

\end{proof}

We already have in this section the key ingredients to do mould calculus: with the help of these results, we will see that computing a conjugating map
will amount to the computation of a character in a quite simple Hopf algebra
(shuffle, quasishuffle or, after arborification--coarborification, in the
Connes--Kreimer Hopf algebra). But this Hopf algebraic structure is already
present when dealing directly with the coefficients of a diffeomorphism and
gives rise to the Fa{\`a} di Bruno Hopf algebra.

\subsection{From $\tmmathbf{G}$ to the Fa{\`a} di Bruno Hopf algebra.}

\subsubsection{A short reminder on Hopf algebras.}

In the sequel we will deal with graded connected Hopf algebras $\mathcal{H}$.
This means first that $\mathcal{H}$ is a graded vector space over
$\mathbbm{C}$
\[ \mathcal{H}= \bigoplus_{n \geqslant 0} \mathcal{H}_n  \]
where moreover $\mathcal{H}_0 \approx \mathbbm{C}$. In order to get a graded
Hopf algebra, $\mathcal{H}$ has to be a graded algebra with
\begin{enumeratenumeric}
  \item A unit $\eta : \mathbbm{C} \rightarrow \mathcal{H}_0$
  
  \item A product $\pi : \mathcal{H} \otimes \mathcal{H} \rightarrow
  \mathcal{H}$
\end{enumeratenumeric}
with the usual commutative diagrams that respectively express the unit and
associativity properties (see \cite{manchon}) and such that $\pi (\mathcal{H}_n
\circ \mathcal{H}_m) \subseteq \mathcal{H}_{n + m}$. It also has to be a
coalgebra with
\begin{enumeratenumeric}
  \item A counit $\varepsilon : \mathcal{H} \rightarrow \mathbbm{C}$
  
  \item A coproduct $\Delta : \mathcal{H} \rightarrow \mathcal{H} \otimes
  \mathcal{H}$
\end{enumeratenumeric}
with the corresponding commutative diagrams for the counit and coassociativity
properties respectively (\cite{manchon}) and such that $\Delta (\mathcal{H}_n)
\subseteq \oplus_{0 \leqslant k \leqslant n} \mathcal{H}_k \otimes
\mathcal{H}_{n - k}$.

With the compatibility relations between the algebra structure and the
coalgebra structure, $\mathcal{H}$ becomes a bialgebra and the graded
structure (with $\mathcal{H}_0 \approx \mathbbm{C}$) ensures that this is a
Hopf algebra: there exists an antipode, namely a linear map $S : \mathcal{H}
\rightarrow \mathcal{H}$ such that:
\[ \pi \circ (\tmop{id} \otimes S) \circ \Delta = \pi \circ (S \otimes
   \tmop{id}) \circ \Delta = \eta \circ \varepsilon \]
Once such a Hopf algebra is given, it induces an algebra structure on \
$\mathcal{L} (\mathcal{H}, \mathbbm{C})$. If $u$ and $v$ are two linear forms,
their convolution product is given by
\[ u \ast v = \pi_{\mathbbm{C}} \circ (u \otimes v) \circ \Delta \]
where $\pi_{\mathbbm{C}}$ is the usual product on $\mathbbm{C}$. Let us
remember that among such morphisms, one can distinguish
\begin{enumeratenumeric}
  \item The {\tmstrong{characters}} (algebra morphisms) that form a group
  \tmtexttt{$\mathcal{C} (\mathcal{H})$} for the convolution, with unit
  $\varepsilon$ and the inverse of a character $\chi$ is given by $\chi \circ
  S$.
  
  \item The {\tmstrong{infinitesimal characters}}, that are the linear
  morphisms $u$ vanishing on $\mathcal{H}_0$ and such that
  \[ u \circ \pi = \pi_{\mathbbm{C}} \circ (u \otimes \varepsilon +
     \varepsilon \otimes u) \]
  this set is a Lie algebra $c (\mathcal{H})$ for the Lie bracket $[u, v] = u
  \ast v - v \ast u$.
\end{enumeratenumeric}
As for vector fields and diffeomorphisms $\mathcal{C} (\mathcal{H})$ behaves
as the Lie group of the Lie algebra $c (\mathcal{H})$ with the log and exp
maps \ :
\[ \begin{array}{ccc}
     \exp_{\ast} (u) & = & \ds \varepsilon + \sum_{s \geqslant 1} \frac{1}{s!}
     u^{\ast^s}\\
     \log_{\ast} (\chi) & = & \ds \sum_{s \geqslant 1} \frac{(- 1)^{s - 1}}{s}
     (\chi - \varepsilon)^{\ast^s}
   \end{array} \]
It is in fact a proalgebraic group, namely an inverse limit of linear
algebraic groups, and the $\exp$ and $\log$ are computed as graded operators
at the level of the homogeneous components ({\cite{ef-manchon_irma}}).

We shall soon see that vector fields and diffeomorphisms can be identified to
infinitesimal characters and characters on Hopf algebra, namely the Fa{\`a} di
Bruno Hopf algebra. But let us first give a concrete example of graded
connected Hopf algebra related to power series.

The coalgebra of coordinates of power series can be defined as follows, for $n
\in \mathbbm{N}^{\nu}$, let us consider the functional :
\[ \begin{array}{ccccc}
     \alpha_n & : & \mathbbm{C} [[x]] & \rightarrow & \mathbbm{C}\\
     &  & f (x) = \sum_{n \in \mathbbm{N}^{\nu}} f_n x^n & \mapsto & \alpha_n
     (f) = f_n
   \end{array} \]
The graded vector space $C = \oplus_{k \geqslant 0} C_k$, where $C_k =
\tmop{Vect}_{\mathbbm{C}} \{\alpha_n, \hspace{1em} | n| = k\}$ is a graded
cocommutative coalgebra for the coproduct induced by the product of series :
\[ \Delta \alpha_n = \sum_{k + l = n} \alpha_k \otimes \alpha_l \hspace{1em}
   (\alpha_n (f.g) = \pi_{\mathbbm{C}} \circ (\Delta \alpha_n) (f \otimes g))
\]
and the counit is given by $\varepsilon (\alpha_0) = 1$ and $\varepsilon
(\alpha_n) = 0$ if $| n| \geqslant 1$. Thanks to this coalgebra structure, the
space $\mathcal{L} (C, \mathbbm{C})$ is a convolution algebra which is
trivially isomorphic to $\mathbbm{C} [[x]]$. In order to define a Hopf
algebra, let us consider the free commutative algebra generated by
$\{\alpha_n, | n| \geqslant 1\}$. By adding a unit $\tmmathbf{1}$, extending
the gradation and the previous coproduct to the product of functionals, one
gets a Hopf algebra $\mathcal{H}$ whose group of characters can be clearly
identified to the group of invertible series :
\[ G^{\tmop{inv}} = \left\{ 1 + \sum_{|n| \geqslant 1} f_n x^n, \hspace{1em}
   f_n \in \mathbbm{C} \right\} \]
The same idea will govern the contruction of the Fa{\`a} di Bruno Hopf algebra
of coordinates on the group $\tmmathbf{G}$.

\subsubsection{The Fa{\`a} di Bruno Hopf algebra.}

The group $\tmmathbf{G}$ is associated to a graded connected algebra. Let us
first remind that
\[ \tmmathbf{G}= \{\varphi (x) = x + u (x), u \in (\mathbbm{C}_{\geqslant 2}
   [[x]])^{\nu} \} \]
where $x = (x_1, \ldots, x_{\nu})$ and $u (x) = (u_1 (x), \ldots, u_{\nu}
(x))$ and we can note
\[ \varphi (x) = (\varphi_i (x))_{1 \leqslant i \leqslant \nu} = \left( x_i
   \left( 1 + \sum_{\eta \in H_i} \varphi^i_{\eta} x^{\eta} \right) \right)_{1
   \leqslant i \leqslant \nu} \]
A diffeomorphism in $\tmmathbf{G}$ is then given by its coefficients
$\varphi^i_{\eta}$, where $i \in \{1, \ldots, \nu\}$ and $\eta \in H_i$,
and for any such couple $(i, \eta)$, one can define functionals on $\tmmathbf{G}$:
\[ \begin{array}{ccccc}
     C^i_{\eta} & : & \tmmathbf{G} & \rightarrow & \mathbbm{C}\\
     &  & \varphi & \mapsto & \varphi^i_{\eta}
   \end{array} \]
Following the same ideas as for $G^{\tmop{inv}}$, the Fa{\`a} di Bruno algebra
is the free commutative algebra generated by the the functionals $C^i_{\eta}$
:
\[ \mathcal{H}_{\tmop{FdB}} =\mathbbm{C} [(C^{i_{}}_{\eta})_{i \in \{1,
   \ldots, \nu\}, \eta \in H_i}] \]
Identifying $\mathbbm{C} \subset \mathcal{H}_{\tmop{FdB}}$ to $\mathbbm{C}.1$,
where 1 is the functional defined on $\tmmathbf{G}$ by $1 (\varphi)$=1, it is
clear that $\mathcal{H}_{\tmop{FdB}}$ acts on $\tmmathbf{G}$, if $P (\ldots,
C^i_{\eta}, \ldots)$ is a polynomial in $\mathcal{H}_{\tmop{FdB}}$, then
\[ P (\ldots, C^i_{\eta}, \ldots) (\varphi) = P (\ldots, \varphi^i_{\eta},
   \ldots) \]
If we define a gradation by $\tmop{gr} (1) = 0$ and
\[ \tmop{gr} (C^{i_1}_{\eta_1} \ldots C^{i_s}_{\eta_s}) = | \eta_1 | + \ldots
   + | \eta_s | \]
then $\mathcal{H}_{\tmop{FdB}}$ is a graded connected commutative algebra.
Now, using the Fa{\`a} di Bruno formulas \ref{FdBf}, it is not difficult to
define a coproduct on this algebra by the relation :
\[ C^i_{\eta} (\varphi \circ \psi) = \pi_{\mathbbm{C}} \circ (\Delta
   C^i_{\eta}) (\varphi \otimes \psi) \]
extended to $\mathbbm{C} [(C^{i_{}}_{\eta})_{i \in \{1, \ldots, \nu\},
\hspace{1em} \eta \in H_i}]$. With this coproduct, $\mathcal{H}_{\tmop{FdB}}$
is a graded connected Hopf algebra.

Now, any diffeomorphism $\varphi$ can be identified to the character (also
noted $\varphi$) on \ $\mathcal{H}_{\tmop{FdB}}$ defined by $\varphi
(C^i_{\eta}) = C^i_{\eta} (\varphi)$ so that $\tmmathbf{G}$ is clearly
isomorphic to $\mathcal{C} (\mathcal{H}_{\tmop{FdB}})$ and, on the same way,
the Lie algebra $\tmmathbf{g}$ is isomorphic to $\tmmathbf{c}
(\mathcal{H}_{\tmop{FdB}})$ (taking into account the due order reversal in the
formulas) . Note that the $\log - \exp$ correspondance between
{\tmstrong{$\tmmathbf{G}$}} and $\tmmathbf{g}$ is exactly the $\log_{\ast} -
\exp_{\ast}$ correspondance between $\mathcal{C} (\mathcal{H}_{\tmop{FdB}})$
and $\tmmathbf{c} (\mathcal{H}_{\tmop{FdB}})$ and the action of $\tmmathbf{g}$
and $\tmmathbf{G}$ on $\mathbbm{C} [[x]]$ corresponds to the coaction $\Phi :
C \rightarrow C \otimes \mathcal{H}_{\tmop{FdB}}$ defined by
\[ \alpha_n (f \circ \varphi) = \pi_{\mathbbm{C}} \circ \Phi (\alpha_n) (f
   \otimes \varphi) \]
which is such that $C$ is a $\mathcal{H}_{\tmop{FdB}}$--comodule coalgebra (cf
\cite{figuer}).

This algebraic work on diffeomorphisms and vector fields may not seem to help
at the present moment, but it will be useful in the sequel and one can already
notice that {\tmem{linearization equations correspond to equations for
characters on}} $\mathcal{H}_{\tmop{FdB}}$.

\subsection{Characters and conjugacy equations.}

\subsubsection{Vector fields.}

Consider a vector field :
\[ X (x) = \sum_{i = 1}^{\nu} X_i (x) \partial_{x_i} \]
such that
\[ X_i (x) = \lambda_i x_i + x_i \sum_{\eta \in \Eta} a^i_{\eta} x^{\eta} =
   \lambda_i x_i + P_i (x)  \hspace{1em} (\lambda_i \in \mathbbm{C}) \]
this vector field can be seen as a perturbation of it linear part
$X^{\tmop{lin}} = \sum \lambda_i x_i \partial_{x_i}$:
\[ X = X^{\tmop{lin}} + P_{} \]
and one would like to know if, through some formal or analytic change of
coordinates $y = \varphi (x)$ ($x = \psi (y)$), the vector field can be
conjugated to its linear part :
\[ \varphi^{\ast} (X) = X^{\tmop{lin}} \hspace{1em} \tmop{or} \hspace{1em}
   \psi^{\ast} (X^{\tmop{lin}}) = X = X^{\tmop{lin}} + P \]
If this is the case, the latter equation reads, for $1 \leqslant i \leqslant
\nu$,
\[ X^{\tmop{lin}} . \psi_i = X_i \circ \psi = \lambda_i \psi_i + P_i \circ
   \psi (x) \]
so
\[ X^{\tmop{lin}} . \psi_i - \lambda_i \psi_i = P_i \circ \psi (x) \]

On one hand, if $\psi_i (x) = x_i \left( 1 + \sum_{\eta \in \Eta} b^i_{\eta}
x^{\eta} \right)$, then
\[ X^{\tmop{lin}} . \psi_i - \lambda_i \psi_i = x_i \sum_{\eta \in H} \langle
   \lambda, \eta \rangle b^i_{\eta} x^{\eta} \]
where, for $\eta = (n_1, \ldots, n_{\nu}) \in H$, $\langle \lambda, \eta
\rangle = \lambda_1 n_1 + \ldots + \lambda_{\nu} n_{\nu}$. From the Hopf
algebra point of view, let $\nabla$ the derivation on
$\mathcal{H}_{\tmop{FdB}}$ defined by \ $\nabla 1 = 0$ and $\nabla C^i_{\eta}
= \langle \lambda, \eta \rangle C^i_{\eta}$, then, if $\chi$ is the character
associated to $\psi$, we have
\[ X^{\tmop{lin}} . \psi_i - \lambda_i \psi_i = \sum_{\eta \in H} \chi \circ
   \nabla (C^i_{\eta}) x^{\eta} \]

On the other hand, If $u$ is the infinitesimal character on
$\mathcal{H}_{\tmop{FdB}}$ defined by $u (1) = 0$ and $u (C^i_{\eta}) =
a^i_{\eta}$, then the conjugacy equation reads
\[ \forall i, \eta, \hspace{1em} \chi \circ \nabla (C^i_{\eta}) = (u \ast
   \chi) (C^i_{\eta}) \]
and, thanks to the fact that $\nabla$ is a derivation and $u$ infinitesimal,
the conjugacy equation reads, on $\mathcal{H}_{\tmop{FdB}}$,
\[ \chi \circ \nabla = u \ast \chi \]
Of course, $\varphi$ is given by the inverse $\chi \circ S$ of \ $\chi$ in
$\mathcal{C} (\mathcal{H}_{\tmop{FdB}}, \mathbbm{C})$.

In other words, one can associate to a vector field $X = X^{\tmop{lin}} + P$
an infinitesimal character $u$, and this vector field is formally conjugated
to $X^{\tmop{lin}}$ if and only if there exists a character $\chi$ such that
the above equation holds. Moreover we have the very classical (\cite{arnold})
result:

\begin{proposition}
  If, for all $\eta \in H$, $\langle \lambda, \eta \rangle \neq 0$ , $X$ is
  formally conjugated to $X^{\tmop{lin}}$.
\end{proposition}

\begin{proof}
  The proof is recursive on the gradation of
  $\mathcal{H}_{\tmop{FdB}}$:
  let $\chi = \ds \sum_{n \geqslant 0} \chi_n$ , where
  $\chi_n$ is the restriction to the $n^{\tmop{th}}$ component of
  $\mathcal{H}_{\tmop{FdB}}$. Note that, necessarily, $\chi_0 = \varepsilon$.
  Let us suppose \ that, for a given $n \geqslant 0$, $\chi_0, \ldots, \chi_n$
  are well-defined and such that, for any monomial $C^{i_1}_{\eta_1} \ldots
  C^{i_s}_{\eta_s}$ with $\tmop{gr} (C^{i_1}_{\eta_1} \ldots
  C^{i_s}_{\eta_s}) = | \eta_1 | + \ldots + | \eta_s | = k \leqslant n$,
  \[ \chi (C^{i_1}_{\eta_1} \ldots C^{i_s}_{\eta_s}) = \chi_k
     (C^{i_1}_{\eta_1} \ldots C^{i_s}_{\eta_s}) = \chi (C^{i_1}_{\eta_1})
     \ldots \chi (C^{i_s}_{\eta_s}) = \chi_{| \eta_1 |} (C^{i_1}_{\eta_1})
     \ldots \chi_{| \eta_s |} (C^{i_s}_{\eta_s}) \]
  Thanks to the definition of $u$, if $C^i_{\eta}$ is in
  $\mathcal{H}_{\tmop{FdB}, n + 1}$ ($| \eta | = n + 1$), then the equation
  reads
  \[ \begin{array}{rcl}
       \langle \lambda, \eta \rangle \chi (C^i_{\eta}) & = & \pi \circ (u
       \otimes \chi) (\Delta (C^i_{\eta}))\\
       & = & u (C^i_{\eta}) \ +\\
       &  &\ds \sum_{s \geqslant 2} \hspace{-3mm} \sum_{\tmscript{\begin{array}{c}
         i_1 = i\\
         1 \leqslant i_2, \ldots, i_s \leqslant \nu\\
         \eta = \eta_1 + \ldots + \eta_s\\
         \eta_k \in H_{i_k}
       \end{array}}} \frac{1}{s - 1!} P^{\eta_1 + e_i}_{i_2, \ldots, i_s}
        u (C^{i_{}}_{\eta_1}) \chi (C^{i_2}_{\eta_2}) \ldots \chi
       (C^{i_s}_{\eta_s})
     \end{array} \]
  Since the right-hand side of this equation is recursively well--defined and $\langle
  \lambda, \eta \rangle$ is nonzero, $\chi (C^i_{\eta})$ is uniquely
  determined. Now, if $C^{i_1}_{\eta_1} \ldots C^{i_s}_{\eta_s}$ is in
  $\mathcal{H}_{\tmop{FdB}, n + 1}$ with $s \geqslant 2$, then, in order to
  get a character, one must have 
  \[\chi (C^{i_1}_{\eta_1} \ldots
  C^{i_s}_{\eta_s}) = \chi (C^{i_1}_{\eta_1}) \ldots \chi (C^{i_s}_{\eta_s}). \]
  But, for $1 \leqslant i \leqslant s$, $| \eta_i | \leqslant n$ and one can
  check that
  \[ \begin{array}{ccc}
       u \ast \chi (C^{i_1}_{\eta_1} \ldots C^{i_s}_{\eta_s}) & = & \ds \sum_{t =
       1}^s \left( \prod_{r \not= t} \chi (C^{i_r}_{\eta_r}) \right) (u \ast
       \chi) (C^{i_t}_{\eta_t})\\
       & = & \ds \sum_{t = 1}^s \left( \prod_{r \not= t} \chi
       (C^{i_r}_{\eta_r}) \right) \langle \lambda, \eta_t \rangle \chi
       (C^{i_t}_{\eta_t})\\
       & = & \ds \langle \lambda, \eta_1 + \ldots + \eta_t \rangle \chi
       (C^{i_1}_{\eta_1} \ldots C^{i_s}_{\eta_s})
     \end{array} \]
  Thus, the equation determines a unique character on
  $\mathcal{H}_{\tmop{FdB}}$. Note, however, that we don't obtain in this
  algebra anything near a closed-form solution. 
\end{proof}

This is the non-resonant case. Note that in the above character equation, one
only needs to be able to compute the values $\chi (C^i_{\eta})$ and assume
that this is a character. Moreover, in $\mathcal{H}_{\tmop{FdB}}$, if we have
geometric estimates on the coefficients $\chi (C^i_{\eta})$ then it is
immediate to conclude on the analyticity of the associated diffeomorphism. But,
in this setting the difficulty lies in the explicit computation of these
coefficients, since the equation $\chi \circ \nabla = u \ast \chi$ involves
the rather complex coproduct of \ $\mathcal{H}_{\tmop{FdB}}$. As we are going
to see next, the same work can be done for diffeomorphisms, with the same
difficulty in the computation.

\subsubsection{Diffeomorphisms.}

Let $l = (l_1, \ldots, l_{\nu}) \in (\mathbbm{C}^{\ast})^{\nu}$ and
$f^{\tmop{lin}}$ defined by $f^{\tmop{lin}} (x_1, \ldots, x_{\nu}) = (l_1 x_1,
\ldots, l_{\nu} x_{\nu})$. For a given analytic diffeomorphism $f$ in
$\tmmathbf{G}$, the diffeomorphism $f^{\tmop{lin}} \circ f$ can be seen as a
perturbation of $f^{\tmop{lin}}$ and one could ask if, at least formally, this
map is conjugate to $f^{\tmop{lin}}$. In other words, does there exist a
diffeomorphism $\varphi \in \tmmathbf{G}$ such that
\[ f^{\tmop{lin}} \circ f \circ \varphi = \varphi \circ f^{\tmop{lin}} \]
or
\[ f \circ \varphi = f^{\tmop{lin}^{- 1}} \circ \varphi \circ f^{\tmop{lin}}
\]
Now if $\xi$ (resp. $\chi$) is the character associated to $f$ (resp.
$\varphi$) then the equation reads
\[ \xi \ast \chi = \chi \circ \sigma \]
where $\sigma (1) = 1$ and $\sigma (C^{i_1}_{\eta_1} \ldots C^{i_s}_{\eta_s})
= l^{\eta_1 + \ldots + \eta_s}$. Once again, it is well-known that if the
diffeomorphism $f^{\tmop{lin}} \circ f$ is non-resonant, i.e.
\[ \forall \eta \in \Eta, \hspace{1em} l^{\eta} - 1 \neq 0 \]
then there exist a unique solution $\chi$ (or $\varphi$). Of course, the proof
follows the same lines as for vector fields.

In both cases, modulo a condition on the linear part, the conjugacy equation
is solvable and can be seen as an equation on characters of
$\mathcal{H}_{\tmop{FdB}}$. Here again, it is {\tmem{in principle}} easy to
obtain the analyticity of a diffeomorphism, with the help of the values of the
character, provided that one can easily compute this character. But, because
of the complexity of the convolution (coproduct in
$\mathcal{H}_{\tmop{FdB}}$), this computation is rather difficult.

The idea of mould calculus amounts to working in much simpler Hopf algebras,
whose coproduct is a deconcatenation coproduct. Then the price to pay is that:
\begin{enumeratenumeric}
  \item One has to make some supplementary condition on the linear part, in
  order that all the coefficients are well defined, namely,
  \[ \forall \eta \in \bar{H}, \hspace{1em} \langle \lambda, \eta \rangle
     \neq 0 \hspace{1em} \tmop{or} \hspace{1em} l^{\eta} - 1 \neq 0 \]
  \item Analyticity becomes hidden.
\end{enumeratenumeric}
We shall see below how arborification cures both plagues.

{\color{black} \section{Mould calculus : a solution to the algebraic
complexity of calculations}\label{s:moulds}

}

One of the most fundamental ideas of mould calculus is to consider
diffeomorphisms of $\tmmathbf{G}$ as series of ``homogeneous'' differential
operators. Focusing on linearization problems, one starts either with
\begin{enumeratenumeric}
  \item one vector fields $X = X^{\tmop{lin}} + P$ where $P$, that belong to
  $\tmmathbf{g}$, can be decomposed in homogeneous components :
  \[ P = \sum_{\eta \in H} \mathbbm{B}_{\eta} \]
  \item one diffeomorphism $f^{\tmop{lin}} \circ f$ with $f$ in
  $\tmmathbf{G}$, whose substitution automorphism can be written :
  \[ \Theta_f = \tmop{Id} + \sum_{\eta \in \bar{H}} \mathbbm{D}_{\eta} \]
\end{enumeratenumeric}
and one has to find a linearization diffeomorphism $\varphi$ whose
substitution automorphism can also decomposed in homogeneous components
\[ \Theta_{\varphi} = \tmop{Id} + \sum_{\eta \in \bar{H}} \mathbbm{F}_{\eta}
\]
it is then natural to try {\tmem{a priori}} to express the components
$\mathbbm{F}_{\eta}$ as (non commutative) polynomials in the original
components delivered by the data of the problem. For example, in the case of
vector fields :
\[ \Theta_{\eta} = \sum_{s \geqslant 1} \sum_{\tmscript{\begin{array}{c}
     \eta_1 + \ldots + \eta_s = \eta\\
     \eta_i \in H
   \end{array}}} M^{\eta_1, \ldots, \eta_s} \mathbbm{B}_{\eta_s} \ldots
   \mathbbm{B}_{\eta_1} \]
In such an expression, convenient properties of symmetry for the coefficients
$M$ will ensure that $\Theta_{\varphi}$ is a substitution automorphism. Doing
so, we have already roughly defined what mould calculus is. Let us focus now
on the Hopf algebras underlying this calculus.

\subsection{Moulds and the concatenation algebra.}

In both types of conjugacy equations, one has to compute an element of
$\tmmathbf{G}$, or, equivalently, a substitution automorphism that can be
decomposed in homogeneous components. But in both case the initial object
already delivers homogeneous components. It seems then reasonable to use them
and their composition to compute the conjugating substitution automorphism.
This suggests to look at the following concatenation algebra : consider $H$ or
$\bar{H}$ as a graded alphabet. Let $\emptyset$ be the empty word. A word will
be noted
\[ \tmmathbf{\eta}= (\eta_1, \ldots, \eta_s) \]
The gradation can be extended to such words (with $| \emptyset | = 0$) and
one can also define the length of a word
\[ l (\tmmathbf{\eta}) = l ((\eta_1, \ldots, \eta_s)) = s \hspace{1em} (l
   (\emptyset) = 0) \]
and its weight
\[ \|\tmmathbf{\eta}\| = \eta_1 + \ldots + \eta_s \in \bar{H} \hspace{1em} (\|
   \emptyset \| = 0) \]
\begin{definition}
  Let $\tmmathbf{H}$ be the set of such words (starting with $H$ or
  $\bar{H}$), then the linear span of $\tmmathbf{H}$, noted $\tmop{Conc}_H$,
  is a graded unital algebra for the concatenation product :
  \[ \forall \tmmathbf{\eta}^1, \tmmathbf{\eta}^2 \in \tmmathbf{H},
     \hspace{1em} \pi (\tmmathbf{\eta}^1 \otimes \tmmathbf{\eta}^2)
     =\tmmathbf{\eta}^1 \tmmathbf{\eta}^2 \]
  where $\tmmathbf{\eta}=\tmmathbf{\eta}^1 \tmmathbf{\eta}^2$ is the usual
  concatenation of the words $\tmmathbf{\eta}^1$ and $\tmmathbf{\eta}^2$.
\end{definition}

Thanks to the gradation, the graded dual of $\tmop{Conc}_H$ is a graded
coalgebra $\tmop{Conc}_H^{\circ}$ whose coproduct is given, on the dual basis
(identified to $\tmmathbf{H}$) by
\[ \Delta (\tmmathbf{\eta}) = \sum_{\tmmathbf{\eta}^1 \tmmathbf{\eta}^2
   =\tmmathbf{\eta}} \tmmathbf{\eta}^1 \otimes \tmmathbf{\eta}^2 \]
and the vector space $\mathcal{L} (\tmop{Conc}_H^{\circ}, \mathbbm{C})$ is an
algebra for the convolution product :
\[ \forall u, v \in \mathcal{L} (\tmop{Conc}_H^{\circ}, \mathbbm{C}),
   \hspace{1em} u \ast v = \pi_{\mathbbm{C}} \circ (u \otimes v) \circ \Delta
\]
In fact, we just defined here the algebra of moulds :

\begin{definition}
  A mould on $H$ (or $\bar{H}$) with values in $\mathbbm{C}$ is a collection
  $M^{\bullet} = \{M^{\tmmathbf{\eta}}, \hspace{1em} \tmmathbf{\eta} \in
  \tmmathbf{H}\}$ of complex numbers.
\end{definition}

It is clear that moulds are in one-to-one correspondance with elements of
$\mathcal{L} (\tmop{Conc}_{\tmmathbf{H}}^{\circ}, \mathbbm{C})$ : since
$\tmmathbf{H}$ is a basis of $\tmop{Conc}_H^{\circ}$, a mould represents the
values of an element of $\mathcal{L} (\tmop{Conc}_H^{\circ}, \mathbbm{C})$ on
the basis $\tmmathbf{H}$ : $M^{\bullet} = \chi (\bullet) =
\{M^{\tmmathbf{\eta}} = \chi (\tmmathbf{\eta}), \hspace{1em} \tmmathbf{\eta}
\in \tmmathbf{H}\}$.

The set of moulds inherits the structure of algebra and for the product, if
$M^{\bullet}$ and $N^{\bullet}$ are two moulds, their product $P^{\bullet} =
M^{\bullet} \times N^{\bullet}$ is given by
\[ P^{\tmmathbf{\eta}} = \sum_{\tmmathbf{\eta}^1 \tmmathbf{\eta}^2
   =\tmmathbf{\eta}} M^{\tmmathbf{\eta}^1} N^{\tmmathbf{\eta}^2} \]
that corresponds to the convolution of the associated morphisms of
$\mathcal{L} (\tmop{Conc}_H^{\circ}, \mathbbm{C})$

\subsection{The underlying Hopf algebras.\label{ss:groups}}

\subsubsection{Vector fields and associated shuffle Hopf algebra.}

As we have seen before, a vector field
\[  X (x) = \sum_{i = 1}^{\nu} X_i (x) \partial_{x_i} \]
such that
\[  X_i (x) = \lambda_i x_i + x_i \sum_{\eta \in \Eta} a^i_{\eta}
      x^{\eta} = \lambda_i x_i + P_i (x)  \hspace{1em} (\lambda_i \in
      \mathbbm{C})  \]
can be decomposed in homogeneous components
\[ X = X^{\tmop{lin}} + \sum_{\eta \in \Eta} \mathbbm{B}_{\eta} \]
with
\[ \mathbbm{B}_{\eta} = \sum_{i = 1}^{\nu} a^i_{\eta} x^{\eta} x_i
   \partial_{x_i} \]
Following Ecalle's convention for the composition of operators, with the help
of these components, one can associate to any word in $\tmmathbf{H}$ a
differential operator in $\mathbbm{C} [x, \partial_x]$ acting on $\mathbbm{C}
[[x]]$ by $\mathbbm{B}_{\emptyset} = \tmop{Id}_{\mathbbm{C} [[x]]}$ and
\[ \forall \tmmathbf{\eta}= (\eta_1, \ldots, \eta_s) \in \tmmathbf{H}/
   \{\emptyset\}, \hspace{1em} \mathbbm{B}_{\tmmathbf{\eta}}
   =\mathbbm{B}_{\eta_s} \ldots \mathbbm{B}_{\eta_1} \]
The family $\mathbbm{B}_{\bullet} = \{\mathbbm{B}_{\tmmathbf{\eta}},
\hspace{1em} \tmmathbf{\eta} \in \tmmathbf{H}\}$ is called a {\tmem{comould}}
and, from a more algebraic point of view, we have the following :

\begin{proposition}
  The map
  \[ \begin{array}{ccccc}
       \rho & : & \tmop{Conc}_H & \rightarrow & \mathbbm{C} [x, \partial_x]\\
       &  & \tmmathbf{\eta} & \mapsto & \mathbbm{B}_{\tmmathbf{\eta}}
     \end{array} \]
  defines an antialgebra morphism (considering the usual composition of
  differential operators)
\end{proposition}

Note also that the action of $\mathbbm{C} [x, \partial_x]$ on a product in
$\mathbbm{C} [[x]]$ defines a coproduct $\Delta : \mathbbm{C} [x, \partial_x]
\rightarrow \mathbbm{C} [x, \partial_x] \otimes \mathbbm{C} [x, \partial_x]$
defined by
\[ \forall u, v \in \mathbbm{C} [[x]], \hspace{1em} \forall D \in \mathbbm{C}
   [x, \partial_x], \hspace{1em} D. (u v) = \pi_{\mathbbm{C} [[x]]} \circ
   \Delta (D) . (u \otimes v) \]
and, as we deal here with vector fields,
\[ \forall \eta \in \Eta, \hspace{1em} \Delta (\mathbbm{B}_{\eta}) =
   \tmop{Id}_{\mathbbm{C} [[x]]} \otimes \mathbbm{B}_{\eta}
   +\mathbbm{B}_{\eta} \otimes \tmop{Id}_{\mathbbm{C} [[x]]} \]
This can be extended to the comould and, using
\begin{enumeratenumeric}
  \item The morphism $\rho$,
  
  \item The operators $L^{\eta}_+$ on $\tmmathbf{H}$, defined by
  \[ L^{\eta}_+ ((\eta_1, \ldots, \eta_s)) = (\eta, \eta_1, \ldots, \eta_s)
     \hspace{1em} (L^{\eta}_+ (\emptyset) = (\eta)) \]
  \[  \]
  and extended by linearity to Conc\tmrsub{H} 
\end{enumeratenumeric}
we get

\begin{theorem}
  With the coproduct defined on the basis $\tmmathbf{H}$ of Conc\tmrsub{H} by
  $\Delta (\emptyset) = \emptyset \otimes \emptyset$ and
  \[ \forall \eta \in \Eta, \hspace{1em} \forall \tmmathbf{\eta} \in
     \tmmathbf{H}, \hspace{1em} \Delta (L_+^{\eta} (\tmmathbf{\eta})) =
     (\tmop{Id} \otimes L_+^{\eta} + L_+^{\eta} \otimes \tmop{Id}) \circ
     \Delta (\tmmathbf{\eta}) \]
  the algebra Conc\tmrsub{H} becomes a graded, cocommutative bialgebra, and
  thus a Hopf algebra whose antipode is given by
  \[ \hspace{1em} \forall \tmmathbf{\eta} \in \tmmathbf{H}, \hspace{1em} S
     (\tmmathbf{\eta}) = (- 1)^{l (\tmmathbf{\eta})} \tmop{rev}
     (\tmmathbf{\eta}) \]
  where $\tmop{rev} (\emptyset) = \emptyset$ and $\tmop{rev} ((\eta_1, \ldots,
  \eta_s)) = (\eta_s, \ldots, \eta_1)$ otherwise. Moreover the morphism $\rho$
  turns to be a coalgebra morphism and, in this case, the comould
  $\mathbbm{B}_{\bullet}$ is said to be cosymmetral..
\end{theorem}

The proof is straightforward.

Going back to the given vector field $X$ and the conjugating equation
$\varphi^{\ast} (X) = X^{\tmop{lin}}$ becomes, in terms of substitution
automorphism,
\[ X. \Theta_{\varphi} = \Theta_{\varphi} .X^{\tmop{lin}} \]
And, since the vector field delivers a family (comould) of differential
operators, it seems reasonable to look for a substitution automorphism
$\Theta_{\varphi}$ that can be written as a mould expansion
\[ \Theta_{\varphi} = \sum_{\tmmathbf{\eta} \in \tmmathbf{H}}
   M^{\tmmathbf{\eta}} \mathbbm{B}_{\tmmathbf{\eta}} \]
where $M^{\bullet} = \{M^{\tmmathbf{\eta}}, \hspace{1em} \tmmathbf{\eta} \in
\tmmathbf{H}\}$ is precisely a mould, with the relevant conditions
(symmetrality) that ensure that $\Theta_{\varphi}$ is a substitution
automorphism. More precisely, since Conc\tmrsub{H} is a graded cocommutative
Hopf algebra, its graded dual is a graded commutative Hopf algebra, noted
$\tmop{Sh}_H$ (for shuffle Hopf algebra on $H$) whose product (resp.
coproduct) is given by the usual shuffle product (resp. deconcatenation
coproduct). But if we consider the group of characters $\mathcal{C}
(\tmop{Sh}_{\Eta}, \mathbbm{C})$ then

\begin{theorem}
  The map
  \[ \begin{array}{ccccc}
       S_{\rho} & : & \mathcal{C} (\tmop{Sh}_{\Eta}, \mathbbm{C}) &
       \rightarrow & \tmmathbf{G}\\
       &  & \chi & \mapsto & \tmop{ev} \left( \sum_{\tmmathbf{\eta} \in
       \tmmathbf{H}} \chi (\tmmathbf{\eta}) \rho (\tmmathbf{\eta}) \right)
     \end{array} \label{th:mouldnorm} \]
  defines a morphism of groups and $\Theta^{\chi} = \sum_{\tmmathbf{\eta} \in
  \tmmathbf{H}} \chi (\tmmathbf{\eta}) \rho (\tmmathbf{\eta})$ is the
  substitution automorphism associated to $S_{\rho} (\chi)$.
\end{theorem}

Note that moulds corresponding to such characters are called
{\tmem{symmetral}} moulds.

\begin{proof}
  Thanks to gradation and homogeneity, it is clear that this defines a
  diffeomorphism of $\tmmathbf{G}$. If $\chi$ is a character, then for two
  power series $u$ \ and $v$,
  \[ \begin{array}{cccc}
       \Theta^{\chi} . (u v) & = & \ds \sum_{\tmmathbf{\eta} \in \tmmathbf{H}}
       \chi (\tmmathbf{\eta}) \rho (\tmmathbf{\eta}) . (u v) & \\
       & = &\ds  \sum_{\tmmathbf{\eta} \in \tmmathbf{H}} \chi (\tmmathbf{\eta})
       \pi_{\mathbbm{C}[ [x]} \circ (\Delta (\rho (\tmmathbf{\eta})) . (u
       \otimes v) & \\
       & = &\ds  \pi_{\mathbbm{C}[ [x]} \circ \left( (\rho \otimes \rho) \left(
       \sum_{\tmmathbf{\eta} \in \tmmathbf{H}} \chi (\tmmathbf{\eta}) \Delta
       (\tmmathbf{\eta}) \right) \right) . (u \otimes v) & \\
       & = & \ds 
       \sum_{\tmmathbf{\eta}, \tmmathbf{\eta}^1, \tmmathbf{\eta}^2 \in
       \tmmathbf{H}} \chi (\tmmathbf{\eta}) \langle \Delta (\tmmathbf{\eta}),
       \tmmathbf{\eta}^1 \otimes \tmmathbf{\eta}^2 \rangle (\rho(\tmmathbf{\eta}^1).u)
      (\rho(\tmmathbf{\eta}^2)).v)  & \\
       & = & \ds 
       \sum_{\tmmathbf{\eta}, \tmmathbf{\eta}^1, \tmmathbf{\eta}^2 \in
       \tmmathbf{H}} \chi (\tmmathbf{\eta}) \langle \tmmathbf{\eta},
       \pi_{\tmop{Sh}_{\Eta}} (\tmmathbf{\eta}^1 \otimes \tmmathbf{\eta}^2)
       \rangle (\rho(\tmmathbf{\eta}^1).u)  (\rho(\tmmathbf{\eta}^2).v)  & \\
       & = & \ds 
       \sum_{\tmmathbf{\eta}^1, \tmmathbf{\eta}^2 \in \tmmathbf{H}} \chi
       (\pi_{\tmop{Sh}_{\Eta}} (\tmmathbf{\eta}^1 \otimes \tmmathbf{\eta}^2)) (\rho(\tmmathbf{\eta}^1).u)  (\rho(\tmmathbf{\eta}^2).v)
        & \\
       & = & \ds 
       \sum_{\tmmathbf{\eta}^1, \tmmathbf{\eta}^2 \in \tmmathbf{H}} (\chi
       (\tmmathbf{\eta}^1)(\rho(\tmmathbf{\eta}^1).u) \chi
       (\tmmathbf{\eta}^2)(\rho(\tmmathbf{\eta}^2).v)&
       \\
       & = & \ds (\Theta^{\chi} .u) (\Theta^{\chi} v) & 
     \end{array} \]
  thus $\Theta^{\chi}$ is a substitution automorphism. Moreover, if $\chi^1$
  and $\chi^2$ are two characters, then
  \[ \Theta^{\chi^1 \ast \chi^2} = \Theta^{\chi^2} . \Theta^{\chi^1} =
     \Theta_{S_{\rho} (\chi^2)} . \Theta_{S_{\rho} (\chi^1)} =
     \Theta_{S_{\rho} (\chi^1) \circ S_{\rho} (\chi^2)} \]
  thus
  \[ S_{\rho} (\chi^1 \ast \chi^2) = S_{\rho} (\chi^1) \circ S_{\rho} (\chi^2)
  \]
  
\end{proof}

Note that $S_{\rho} (\mathcal{C}(\tmop{Sh}_{\Eta}, \mathbbm{C}))$ maybe only
be \ a subgroup of {\tmstrong{$\tmmathbf{G}$}} but, in linearization
equations, it is reasonable to look for the change of coordinate in this
subgroup. Suppose that $\varphi^{\ast} (X) = X^{\tmop{lin}}$ ($X.
\Theta_{\varphi} = \Theta_{\varphi} .X^{\tmop{lin}}$) where $\varphi =
S_{\rho} (\chi)$. If $u$ is the infinitesimal character on $\tmop{Sh}_{\Eta}$
defined by
\[ u (\tmmathbf{\eta}) = \left\{ \begin{array}{lll}
     1 & \tmop{if} & l (\tmmathbf{\eta}) = 1\\
     0 & \tmop{if} & l (\tmmathbf{\eta}) \not= 1
   \end{array} \right. \]
then,
\[ X = X^{\tmop{lin}} + \sum_{\tmmathbf{\eta} \in \tmmathbf{H}} u
   (\tmmathbf{\eta}) \rho (\tmmathbf{\eta}) \]
and
\[ X. \Theta_{\varphi} = X^{\tmop{lin}} . \Theta_{\varphi} +
   \sum_{\tmmathbf{\eta} \in \tmmathbf{H}} \chi \ast u (\tmmathbf{\eta}) \rho
   (\tmmathbf{\eta}) = \Theta_{\varphi} .X^{\tmop{lin}} \]
Since
\[ [X^{\tmop{lin}}, \mathbbm{B}_{\eta}] = \langle \lambda, \eta \rangle
   \mathbbm{B}_{\eta} \]
we have
\[ X^{\tmop{lin}} . \Theta_{\varphi} - \Theta_{\varphi} .X^{\tmop{lin}} =
   \sum_{\tmmathbf{\eta} \in \tmmathbf{H}} (\nabla \chi) (\tmmathbf{\eta})
   \rho (\tmmathbf{\eta}) \]
where
\[ (\nabla \chi) (\tmmathbf{\eta}) = \langle \lambda, \|\tmmathbf{\eta}\|
   \rangle \chi (\tmmathbf{\eta}) \]
so that the conjugacy equation can be turned into a character equation
\[ \nabla \chi + \chi \ast u = 0 \]
For the inverse character $\xi$, corresponding to the inverse of the
diffeomorphism, we get,
\[ \nabla \xi = u \ast \xi \]
\begin{proposition}
  Under the assumption that for any $\eta$ in $\bar{H}$, $\langle \lambda,
  \eta \rangle \not= 0$, the above equation determines a unique symmetral
  mould (character) whose values are
  \[ \xi (\eta_1, \ldots, \eta_s) = \frac{1}{\langle \lambda, \eta_1 + \ldots
     + \eta_s \rangle \langle \lambda, \eta_2 + \ldots + \eta_s \rangle \ldots
     \langle \lambda, \eta_s \rangle} \]
  and its inverse is given by
  \[ \chi (\eta_1, \ldots, \eta_s) = \frac{(- 1)^s}{\langle \lambda, \eta_1
     \rangle \langle \lambda, \eta_1 + \eta_2 \rangle \ldots \langle \lambda,
     \eta_1 + \ldots + \eta_s \rangle} \]
\end{proposition}

The proof is straightforward. For example, $\xi (\emptyset) = 1$ and, for
$\eta \in H$ and $\tmmathbf{\eta} \in \tmmathbf{H}$, the equation reads
\[ \langle \lambda, \eta +\|\tmmathbf{\eta}\| \rangle \xi (L_+^{\eta}
   (\tmmathbf{\eta})) = (u \ast \xi) (L_+^{\eta} (\tmmathbf{\eta}))) = u
   (\eta) \xi (\tmmathbf{\eta}) = \xi (\tmmathbf{\eta}) \]
One can check that this is a character (symmetral mould) and $\chi$ can be
either computed directly or as the inverse of the character $\xi$. In the
latter case, since $\xi$ is a character,
\[ \chi = \xi^{\ast^{- 1}} = \xi \circ S \]
where the antipode $S$ in $\tmop{Sh}_H$ is given by
\[ S (\eta_1, \ldots, \eta_s) = (- 1)^s (\eta_s, \ldots, \eta_1) \]

Now the same can be done for the linearization of diffeomorphisms.

\subsubsection{Diffeomorphisms and the associated quasishuffle Hopf algebra.}

Once again, let $l = (l_1, \ldots, l_{\nu}) \in (\mathbbm{C}^{\ast})^{\nu}$
and $f^{\tmop{lin}}$ defined by $f^{\tmop{lin}} (x_1, \ldots, x_{\nu}) = (l_1
x_1, \ldots, l_{\nu} x_{\nu})$. For a given analytic diffeomorphism $f$ in
$\tmmathbf{G}$, the diffeomorphism $f^{\tmop{lin}} \circ f$ can be seen as a
perturbation of $f^{\tmop{lin}}$ and one could ask if, at least formally, this
map is conjugated to $f^{\tmop{lin}}$. In other words, does there exist a
diffeomorphism $\varphi \in \tmmathbf{G}$ or $\varphi \in
\tmmathbf{G}_{\tmop{ana}}$ such that
\[ f^{\tmop{lin}} \circ f \circ \varphi = \varphi \circ f^{\tmop{lin}} \]
If we define on $\mathbbm{C} [[x]]$ the operator $F^{\tmop{lin}}$ by
$F^{\tmop{lin}} .u = u \circ f^{\tmop{lin}}$, then the equation becomes
\[ F_{\varphi} .F_f .F^{\tmop{lin}} = F^{\tmop{lin}} .F_{\varphi} \]
As for vector fields, the substitution automorphims $F_f$ is a series of
homogeneous differential operators :
\[ F_f = \tmop{Id}_{\mathbbm{C} [[x]]} + \sum_{\eta \in \bar{H}}
   \mathbbm{D}_{\eta} \]
and, as in the previous section, the map
\[ \begin{array}{ccccc}
     \rho & : & \tmop{Conc}_{\bar{H}} & \rightarrow & \mathbbm{C} [x,
     \partial_x]\\
     &  & \tmmathbf{\eta}= (\eta_1, \ldots, \eta_s) & \mapsto &
     \mathbbm{D}_{\tmmathbf{\eta}} =\mathbbm{D}_{\eta_s} \ldots
     \mathbbm{D}_{\eta_1}
   \end{array} \]
defines an anti algebra morphism (with $\rho (\emptyset)
=\mathbbm{D}_{\emptyset} = \tmop{Id}_{\mathbbm{C} [[x]]}$). Now the main
difference with vector fields is that the definition of $\rho$ is based on
homogeneous components of a substitution automorphism, for which we have :
\[ \forall \eta \in \bar{H}, \hspace{1em} \Delta (\mathbbm{D}_{\eta}) =
   \tmop{Id}_{\mathbbm{C} [[x]]} \otimes \mathbbm{D}_{\eta} + \sum_{\eta_1 +
   \eta_2 = \eta} \mathbbm{D}_{\eta_1} \otimes \mathbbm{D}_{\eta_2}
   +\mathbbm{D}_{\eta} \otimes \tmop{Id}_{\mathbbm{C} [[x]]} \]
But if we define $\Delta (\emptyset) = \emptyset \otimes \emptyset$ and, for
$\eta \in \bar{H}$,
\[ \Delta ((\eta)) = \emptyset \otimes (\eta) + \sum_{\eta_1 + \eta_2 = \eta}
   (\eta_1) \otimes (\eta_2) + (\eta) \otimes \emptyset \]
then, extending this coproduct to $\tmop{Conc}_{\bar{H}}$, we get

\begin{theorem}
  With this coproduct, the algebra $\tmop{Conc}_{\bar{H}}$ is a graded,
  cocommutative bialgebra, and thus a Hopf algebra. Moreover the morphism
  $\rho$ is a coalgebra morphism.
\end{theorem}

The proof is quite trivial since this Hopf algebra is the graded dual of a
classical quasishuffle Hopf algebra noted $\tmop{QSh}_{\bar{H}}$ (for
quasishuffle Hopf algebra on $\bar{H}$, see \cite{hoffman}) whose product
(resp. coproduct) is given by the usual quasishuffle product (resp.
deconcatenation coproduct). And, once again,

\begin{theorem}
  The map
  \[ \begin{array}{ccccc}
       S_{\rho} & : & \mathcal{C} (\tmop{QSh}_{\bar{H}}, \mathbbm{C}) &
       \rightarrow & \tmmathbf{G}\\
       &  & \chi & \mapsto & \tmop{ev} \left( \sum_{\tmmathbf{\eta} \in
       \tmmathbf{H}} \chi (\tmmathbf{\eta}) \rho (\tmmathbf{\eta}) \right)
     \end{array} \]
  defines a morphism of groups and $F^{\chi} = \sum_{\tmmathbf{\eta} \in
  \tmmathbf{H}} \chi (\tmmathbf{\eta}) \rho (\tmmathbf{\eta})$ is the
  substitution automorphism associated to $S_{\rho} (\chi)$.
\end{theorem}

The proof is the same as above. Once again a mould
$M^{\bullet} = \{M^{\tmmathbf{\eta}}, \hspace{1em} \tmmathbf{\eta} \in
\tmmathbf{H}\}$ defines a linear map from $\tmop{QSh}_{\bar{H}}$ to
$\mathbbm{C}$ and this mould is
\begin{itemizedot}
  \item \tmtextit{symmetrel} if the associated morphism is in $\mathcal{C}
  (\tmop{QSh}_{\bar{H}}, \mathbbm{C})$,
  
  \item \tmtextit{alternel} if the associated morphism is in $c
  (\tmop{QSh}_{\bar{H}}, \mathbbm{C})$.
\end{itemizedot}

Going back to the linearization equation
\[ \Theta_{\varphi} . \Theta_f . \Theta^{\tmop{lin}} = \Theta^{\tmop{lin}} .
   \Theta_{\varphi} \]
with
\[ \Theta_f = \tmop{Id}_{\mathbbm{C} [[x]]} + \sum_{\eta \in \bar{H}}
   \mathbbm{D}_{\eta} \]
\begin{enumeratenumeric}
  \item $f = S_{\rho} (\xi)$ where $\xi$ is the character defined by $\xi
  (\emptyset) = 1$ and, for $s \geqslant 1$,
  \[ \xi ((\eta_1, \ldots, \eta_s)) = \left\{ \begin{array}{lll}
       1 & \tmop{if} & s = 1\\
       0 & \tmop{if} & s \geqslant 2
     \end{array} \right._{} \]
  \item If there exists a character $\chi$ such that
  \begin{equation}
    \chi \circ \sigma = \xi \ast \chi \hspace{1em} (\sigma (\tmmathbf{\eta}) =
    l^{\|\tmmathbf{\eta}\|} \tmmathbf{\eta})
  \end{equation}
\end{enumeratenumeric}
then $\varphi = S_{\rho} (\chi)$ is a solution to the linearization equation.
Finally, we have the following:

\begin{proposition}
  Under the assumption that for any $\eta$ in $\bar{H}$, $l^{\eta} \not= 1$,
  the above equation determines a unique symmetral mould (character) whose
  values are
  \[ \chi (\eta_1, \ldots, \eta_s) = \frac{1}{(l^{\eta_1 + \ldots + \eta_s} -
     1) (l^{\eta_2 + \ldots + \eta_s} - 1) \ldots (l^{\eta_s} - 1)} \]
  and its inverse is given by
  \[ \chi^{\ast^{- 1}} (\eta_1, \ldots, \eta_s) = \frac{(- 1)^s}{(l^{\eta_1} -
     1) \ldots (l^{\eta_1 + \ldots + \eta_{s - 1}} - 1) (l^{\eta_1 + \ldots +
     \eta_s} - 1)} \]
\end{proposition}

These have been known for a long time, using mould calculus (see {\cite{snag}}).
As for vector fields, we have $\chi (\emptyset) = 1$ and, for $\eta \in H$ and
$\tmmathbf{\eta} \in \tmmathbf{H}$, the equation for the linearization
character reads:
\[ l^{\eta + \|\tmmathbf{\eta}\|} \chi (L_+^{\eta} (\tmmathbf{\eta})) = (\xi
   \ast \chi) (L_+^{\eta} (\tmmathbf{\eta}))) = \xi (\eta) \chi
   (\tmmathbf{\eta}) + \chi (L_+^{\eta} (\tmmathbf{\eta})) = \chi
   (\tmmathbf{\eta}) + \chi (L_+^{\eta} (\tmmathbf{\eta})) \]
One can check that this is a character (symmetrel mould) and $\chi^{\ast^{-
1}}$ can be either computed directly or as the inverse of the character
$\chi$. In the latter case, since $\chi$ is a character,
\[ \chi^{\ast^{- 1}} = \chi \circ S \]
and the antipode in $\tmop{QSh}_{\bar{H}}$ is also given by
\[ S (\eta_1, \ldots, \eta_s) = (- 1)^s \sum_{\tmmathbf{\eta}=
   (\tmmathbf{\eta}^1_{}  \ldots  \tmmathbf{\eta}^t)} ( \|
   \tmmathbf{\eta}^t \|, \ldots, \| \tmmathbf{\eta}^1 \|) \]
{\color{red} }

(the sum involves all the decompositions of the sequence $\tmmathbf{\eta}$ by
concatenation of non-empty subsequences $\tmmathbf{\eta}^i$).

\subsection{Analyticity and the need for some intermediate Hopf algebras.}

To sum up the previous sections, under some algebraic condition on $\lambda$
or $l$, one can perform the linearization with the help of a formal
diffeomorphism, whose substitution automorphism is given by a character $\chi$
:
\[ \Theta_{\varphi} = \sum_{\tmmathbf{\eta} \in \tmmathbf{H}} \chi
   (\tmmathbf{\eta}) \rho (\tmmathbf{\eta}) \]
Under some classical diophantine condition, we shall prove below that such
characters have a geometric growth, meaning that an estimate of the following
type is satisfied ($C$ being a constant):
\[ | \chi (\tmmathbf{\eta}) | \leqslant C^{\tmop{gr} (\tmmathbf{\eta})} \]
so that one could hope that the associated diffeomorphism will be analytic.
However, this kind of estimates are not sufficient. The reason is the
following : if
\[ \varphi_i (x) = x_i \left( 1 + \sum_{\eta \in \Eta} a^i_{\eta} x^{\eta}
   \right) \]
then
\[ x_i a^i_{\eta} x^{\eta} = \sum_{\|\tmmathbf{\eta}\| = \eta} \chi
   (\tmmathbf{\eta}) \rho (\tmmathbf{\eta}) .x_i \]
and the coefficient in $\rho (\tmmathbf{\eta}) .x_i$ tends to grow factorially
with the length of $\tmmathbf{\eta}$, an inevitable feature if we try to bound
by brute force the size of the composition of $r$ ordinary differential
operators : some $r$! factors appear. For example, in dimension 1,
\[ (t^2 \partial_t)^r .t = (r - 1) !t^{r + 1} \]
But, on the other hand, this does not mean that the diffeomorphism is
divergent : many terms contribute to a same power of $x$ and some
{\tmem{compensations}} may arise. Indeed this is the case and, surprisingly,
this compensation phenomenon can be taken into account, using the so-called
arborification--coarborification process which, algebraically, relies on the
use of the Connes--Kreimer Hopf algebra, as we shall see next. \

In fact, the situation can profitably be described as such:

-- Direct calculations at the level of diffeomorphisms immediately translate
into recursive relations in the Fa{\`a} di Bruno Hopf algebra, which are difficult to solve
because the coproduct in $\mathcal{H}_{\tmop{FdB}}$ is a complicated one, a complexity that mirrors
the Fa{\`a} di Bruno formula for the computation of the $n^{\tmop{th}}$ coefficient of the
composition of 2 formal series.

-- Mould--comould expansions, on the contrary, lead to simple equations for
moulds (be that in the symmetral or the symmetrel case); this simplicity is
itself an image of the simplicity of the coproducts of the shuffle or
quasishuffle Hopf algebras. These equations yield in fact {\tmem{closed--form
expressions}} for the sought moulds, which are {\tmem{surprisingly
explicit}}.Yet, when one wants to go beyond the formal level, to eventually
get analytic transformations, these expressions are too coarse: although its
is usually relatively easy to prove geometric growth estimates based on the
explicit mould formulas, the inevitable factorial growth that composing
differential operators brings along, is an obstacle to convergence.

There is thus a need for some {\tmstrong{intermediate Hopf algebra}} for which
the calculations are still tractable, yet efficient enough to yield analytic
functions when needed. This is exactlty what arborification--coarborification
does and, in terms of Hopf algebras, the decorated Connes--Kreimer algebra
will then rather naturally enter the stage.

\section{Arborification--Coarborification.}\label{s:arbor}

\subsection{Hopf algebras of trees.}

We use here the results and notations developped in {\cite{ck}},
{\cite{foissy1}} and {\cite{foissy2}}. A (non--planar) rooted tree $T$ is a
connected and simply connected set of oriented edges and vertices such that
there is precisely one distinguished vertex (the root) with no incoming edge.
An alternative definition can be given in terms of posets containing a
smallest element, and for which each element has at most one predecessor. A
forest $F$ is a monomial in rooted trees. Let $l (F)$ be the number of
vertices in $F$. Using the set $H$ we can decorate a forest, that is to say
that, to each vertex $v$ of $F$, we associate an element $h (v)$ of $H$. We
note $\mathcal{T}_H$ (resp. $\mathcal{F}_H$) the set of decorated trees (resp.
forests) that contains the empty tree noted $\emptyset$. In fact there is a
natural equivalence relation for trees, two trees being equivalent iff there
is an automorphism of decorated posets that sends one to the other. It is
rather the set of equivalent classes of trees that is denoted by
$\mathcal{T}_H$, using a traditional abuse of language. As for sequences, if a
forest $F$ is decorated by $\eta_1, \ldots, \eta_s$ ($l (F) = s$), we note
\[ \|F\| = \eta_1 + \ldots + \eta_s \in \bar{H}, \hspace{1em} \tmop{gr} (F) =
   \tmop{gr} (\eta_1) + \ldots \tmop{gr} (\eta_s) \]
For example, if
\[ T = \BeunedeuxSetroisequatre
\]
then $l (T) = 4$ and $\|T\| = \eta_1 + \eta_2 + \eta_3 + \eta_4$.

Let us also recall that, for $\eta$ in $H$, the operator $B_{\eta}^+$
associates to a forest of decorated trees the tree with root decorated by
$\eta$ connected to the roots of the forest : $B_{\eta}^+ (\emptyset)$ is the
tree with one vertex decorated by $\eta$ and for example :

 \[ B_{\eta}^+ \Big ( \BeunedeuxSetroisequatre \quad \Secinqesix  \Big ) = \BeBeunedeuxSetroisequatreSecinqesix
\]

The linear span $\tmop{CK}_H$ of $\mathcal{F}_H$ is a graded commutative
algebra for the product
\[ \pi (F_1 \otimes F_2) = F_1 F_2 \]
and the unit $\emptyset$. Moreover, with the coproduct $\Delta$ given by
induction by $\Delta (\emptyset) = \emptyset \otimes \emptyset$, $\Delta (T_1
\ldots T_k) = \Delta (T_1) \ldots \Delta (T_k)$ and
\[ \Delta (B_{\eta}^+ (F)) = B_{\eta}^+ (F) \otimes \emptyset + (\tmop{Id}
   \otimes B_{\eta}^+) \circ \Delta (F) \]
$\tmop{CK}_H$ is the Connes-Kreimer Hopf algebra of trees decorated by $H$.

There exists a combinatorial description of this coproduct (see
{\cite{foissy1}}). For a given tree $T \in \mathcal{T}_H$, an admissible cut
$c$ is a subset of its vertices such that, on the path from the root to an
element of $c$, no other vertex of $c$ is encountered. For such an admissible
cut, $P^c (T)$ is the \ product of the subtrees of $T$ whose roots are in $c$
and $R^c (T)$ is the remaining tree, once these subtrees have been removed.
With these definitions, for any tree $T$, we have
\[ \Delta (T) = \sum_{ \tmop{adm}\  \tmop{cut}} P^c (T) \otimes R^c (T) \]
For example,
\begin{eqnarray*}
 \Delta  \Big ( \Ceunedeuxetrois \Big ) & = & \Ceunedeuxetrois \otimes \   \emptytree 
  + \quad \etrois  \otimes  \Seunedeux
+ \edeux \otimes  \Seunetrois \\
& & \quad + \quad \edeuxetrois \otimes  \eun  +  \emptytree \  \otimes   \Ceunedeuxetrois 
 \end{eqnarray*}

Once again we can consider the convolution algebra $\mathcal{L} (\tmop{CK}_H,
\mathbbm{C})$ and any morphism $u$ of this algebra is given by its values on
the basis $\mathcal{F}_H$. The definitions of arborescent moulds can then be
rephrased :

\begin{definition}
  An arborescent mould $M^{\bullet^<}$ on $H$ with values in $\mathbbm{C}$ is
  a collection of complex numbers$\{M^F \in \mathbbm{C} \nocomma, \hspace{1em}
  F \in \mathcal{F}_H \}$.
\end{definition}

Such arborescent moulds are in one to one correspondance with the elements of
$\mathcal{L} (\tmop{CK}_H, \mathbbm{C})$ and the product of such moulds
corresponds to the convolution of the associated linear morphism.

Note that
\begin{enumeratenumeric}
  \item a character on $\tmop{CK}_H$ defines a \tmtextit{separative} mould
  $M^{\bullet^<}$, i.e.
  \[ M^{T_1 \ldots T_s} = M^{T_1} \ldots M^{T_s} \hspace{1em} (\tmop{and}
     M^{\emptyset} = 1) \]
  \item an infinitesimal character on $\tmop{CK}_H$ defines a
  \tmtextit{antiseparative} mould $M^{\bullet^<}$, i.e. for $s \geqslant 2$,
  \[ M^{T_1 \ldots T_s} = 0 \hspace{1em} (\tmop{and} M^{\emptyset} = 0) \]
\end{enumeratenumeric}
Since the coproduct is not as trivial as before, the convolution and inversion
of characters are not so easy to handle. Nonetheless, we get partial but
useful formulas for ``root'' characters, namely characters vanishing on
forests $T_1 \ldots T_s$ such that at least one of the trees $T_i$ as more
than one vertex. For such a character $\chi$, we have that
\[ \forall u \in \mathcal{L} (\tmop{CK}_H, \mathbbm{C}), \hspace{1em} \forall
   T = B^+_{\eta} (F) \in \mathcal{T}_H, \hspace{1em} (u \ast \chi) (T) = u
   (T) + u (F) \chi (\bullet_{\eta}) \]
and one can deduce that for any tree $T$ decorated by $\eta_1, \ldots, \eta_s$
($l (T) = s$)
\[ \hspace{1em} \chi^{\ast^{- 1}} (T) = (- 1)^{l (T)} \chi (\bullet_{\eta_1})
   \ldots \chi (\bullet_{\eta_s}) \]

The graded dual of $\tmop{CK}_H$ will play a crucial role in the sequel and
is strongly related to the Grossman-Larson Hopf algebra
$\mathcal{\tmop{GL}}_H$ (see {\cite{gl1}}, {\cite{gl2}}, {\cite{hoffman}} and
{\cite{zhao1}}). The algebra $\mathcal{\tmop{GL}}_H$ is the linear span of
rooted trees whose vertices (except the root) are decorated by $H$ (see
{\cite{foissy2}}) : using $0$ to note the absence of decoration, any such tree
can be written $B_0^+ (F)$ where $F$ is in $\mathcal{F}_H$.

Let $F = T_1 \ldots T_k \in \mathcal{F}_H$ and $T_0 \in B_0^+ (\mathcal{F}_H)$, the product of $B^+_0(F)$ and $T_0$ in $\mathcal{\tmop{GL}}_H$ is defined as follows: for any sequence $\tmmathbf{s} = (s_1, \ldots s_k)$ of vertices of $T_0$ (with possible repetitions), let $(T_1, \ldots, T_k) \circ_{\tmmathbf{s}} T_0$ be
the tree of $B^+_0 (F)$ obtained by identifying the root of $B_0^+(T_i)$ with the
vertex $s_i$ in $T_0$. The product $\pi$ in $\mathcal{\tmop{GL}}_H$ is then defined by
\[ B_0^+ (T_1 \ldots T_k) .T_0 = \sum_{\tmmathbf{s}} (T_1, \ldots, T_k)
   \circ_{\tmmathbf{s}} T_0 \]
and the unit is $B_0^+ (\emptyset)$. The coproduct is given by
\[ \Delta (B^+_0 (T_1 \ldots T_k)) = \sum_{I \subseteq \{1, \ldots, k\}} B_0^+
   (T_I) \otimes B_0^+ (T_{\{1, \ldots k\} - I}) \]
Where $I$ is any subset of $\{1, \ldots, k\}$ and $T_I = \prod_{i \in I} T_i$.

For a forest $F$ in $\mathcal{F}_H$ we remind that the \ symmetry factor of
$F$ is defined by :
\begin{enumeratenumeric}
  \item $s ((\eta)) = 1$ ;
  
  \item $s (B^+_{\eta} (F)) = s (F)$ ;
  
  \item $s (T_1^{a_1} \ldots T_k^{a_k}) = s (T_1)^{a_1} \ldots s (T_k)^{a_k}
  a_1 ! \ldots a_k$! if $T_1$,...,$T_k$ are {\tmem{distinct}} rooted trees.
\end{enumeratenumeric}
This factor $s (F)$ \ is the cardinal of the group of automorphisms of the
decorated poset $F$.

We have the following result, which is by now a classical one, and for which
various proofs are available ({\cite{foissy2}}, {\cite{kreimer_chen}},
{\cite{hoffman}}, {\cite{zhao1}}).

\begin{lemma}
  \label{Foissy} The map $\phi$ from $\tmop{GL}_H$ to $\tmop{CK}_H^{\circ}$
  defined by $\phi (B^+_0(F)) = s_F F$ defines an isomorphism of graded Hopf
  algebras between $\tmop{GL}_H$ and $\tmop{CK}_H^{\circ}$.
\end{lemma}

\subsection{Homogeneous coarborification.}

In each case (Vector Fields or Diffeomorphisms), the initial object defines a
morphism $\rho$ from $\tmop{Sh}_{\Eta}^{\circ}$ or $\tmop{Qsh}_H^{\circ}$ to
$\mathbbm{C} [x, \partial_x]$ which is a coalgebra morphism and an algebra
antimorphism that allows to compute some diffeomorphisms as characters on
$\tmop{Sh}_{\Eta}^{}$ or $\tmop{Qsh}_H^{}$. We will essentially follow the
same lines but with a morphism $\rho^<$ from $\tmop{CK}^{\circ}_H$ to
$\mathbbm{C} [x, \partial_x]$. Starting with this map $\rho$, \ one can
define, using Ecalle's homogeneous coarborification the following linear
morphism :

\begin{definition}
  The linear morphism $\rho^<$ from $\tmop{CK}^{\circ}_H$ to $\mathbbm{C} [x,
  \partial_x]$ is defined on its linear basis by the following rules
  \begin{enumeratenumeric}
    \item $\rho^< (\emptyset) = \tmop{Id}$,
    
    \item If $T = B^+_{\eta} (F)$ is a non empty tree, then
    \[ \rho^< (T) = \sum_{i = 1}^{\nu} (\rho^< (F) . (\rho (\eta) .x_i))
       \partial_{x_i} \]
    \item If $F = T_1 \ldots T_s$ with $s \geqslant 2$, then
    \[ \rho^< (F) = \frac{1}{d_1 ! \ldots d_k !} \sum_{1 \leqslant i_1,
       \ldots, i_s \leqslant \nu} (\rho^< (T_1) .x_{i_1}) \ldots (\rho^< (T_s)
       .x_{i_s}) \partial_{x_{i_1}} \ldots \partial_{x_s} \]
    where $F = T_1 \ldots T_s$ is the product of $k$ distinct decorated trees,
    with multiplicities $d_1, \ldots, d_k$ ($d_1 + \ldots + d_k = s$).
  \end{enumeratenumeric}
\end{definition}

From this recursive definition, one already see that the differential operator
$\rho^< (F)$ is of order $r (F)$ (number of roots) and of homogeneity $\|F\|$.
Thanks to the order of $\rho^< (F)$, this morphism is a coalgebra morphism and
we have in fact the following:

\begin{theorem}
  $\rho^<$ is a Hopf morphism.
\end{theorem}

\begin{proof}
  The proof is based on the following result of Grossman and Larson (see
  {\cite{zhao1}}, {\cite{zhao2}}): Let $\tau$ the map from $\tmop{GL}_H$ to
  $\mathbbm{C} [x, \partial_x]$ defined by
  \begin{enumeratenumeric}
    \item $\tau (B^+_0 (\emptyset)) = \tmop{Id}$ ,
    
    \item If $T = B^+_0 (t)$ where $t = B^+_{\eta} (t_1 \ldots t_s)$ is a tree
    of $\mathcal{F}_H$, then
    \[ \tau (T) = \sum_{i = 1}^{\nu} (\tau (B^+_0 (t_1 \ldots t_s)) . (\rho
       (\eta) .x_i)) \partial_{x_i} \]
    \item If $T = B^+_0 (t_1 \ldots t_s)$ ($s \geq 2$), then
    \[ \tau (B^+_0(t_1 \ldots t_s)) = \sum_{1 \leqslant i_1, \ldots, i_s
       \leqslant \nu} (\tau (B^+_0(t_1)) .x_{i_1}) \ldots (\tau (B^+_0(t_s))
       .x_{i_s}) \partial_{x_{i_1}} \ldots \partial_{x_{i_s}} \]
  \end{enumeratenumeric}
  Then $\tau$ is a Hopf morphism (the differential operators thus recursively defined are also known as elementary differentials in the literature on B--series, etc). One can convince oneself with the following
  example where :
  \[ T_1 =  \Szeroeun
 \hspace{1em} T_2 = \Czeroedeuxetrois
\nocomma,\quad  \pi (T_1 \otimes
     T_2) =  \Czeroeunedeuxetrois+ \BzeroetroisSedeuxeun
+ \BzeroedeuxSetroiseun \]
  We have :
  \[ \tau (T_1) = \sum_{i_1 = 1}^{\nu} (\rho (\eta_1) .x_{i_1})
     \partial_{x_{i_1}} \hspace{1em}, \hspace{1em} \tau (T_2) = \sum_{i_2, i_3
     = 1}^{\nu} (\rho (\eta_2) .x_{i_2}) (\rho (\eta_3) .x_{i_3})
     \partial_{x_{i_2}} \partial_{x_{i_3}} \]
  and, using Leibniz rule,
  \[ \begin{array}{ccc}
       \tau (T_1) . \tau (T_2) & = & \ds \left( \sum_{i_1 = 1}^{\nu} (\rho
       (\eta_1) .x_{i_1}) \partial_{x_{i_1}} \right)  \left( \sum_{i_2, i_3 =
       1}^{\nu} (\rho (\eta_2) .x_{i_2}) (\rho (\eta_3) .x_{i_3})
       \partial_{x_{i_2}} \partial_{x_{i_3}} \right) \\
       & = & \ds \left( \sum_{i_1, i_2, i_3 = 1}^{\nu} (\rho (\eta_1) .x_{i_1})
       (\rho (\eta_2) .x_{i_2}) (\rho (\eta_3) .x_{i_3}) \partial_{x_{i_1}}
       \partial_{x_{i_2}} \partial_{x_{i_3}} \right.\\
       &  & \ds + \sum_{i_1, i_2, i_3 = 1}^{\nu} \left( (\rho (\eta_1) .x_{i_1})
       (\partial_{x_{i_1}} (\rho (\eta_2) .x_{i_2})) \right) (\rho (\eta_3)
       .x_{i_3}) \partial_{x_{i_2}} \partial_{x_{i_3}}\\
       &  & \ds \left. + \sum_{i_1, i_2, i_3 = 1}^{\nu} (\rho (\eta_2) .x_{i_2})
       \left( (\rho (\eta_1) .x_{i_1}) (\partial_{x_{i_1}} (\rho (\eta_3)
       .x_{i_3})) \right) \partial_{x_{i_2}} \partial_{x_{i_3}} \right)\\
       & = & \ds \tau \Big ( \Czeroeunedeuxetrois \Big) + \tau \Big ( \BzeroetroisSedeuxeun \Big ) + \tau \Big (  \BzeroedeuxSetroiseun \Big )
     \end{array} \]
  But, thanks to the recursive definition of $\rho^<$, $\tau$ and $\phi$, we
  have $\rho^< = \tau \circ \phi^{- 1}$ and , since both $\tau$ and $\phi^{-
  1}$ (see \cite{foissy1, foissy2}) are Hopf morphisms, so is $\rho^<$.
\end{proof}

Note that the construction of $\tau$ was given by Grossman and Larson only for
the case of a family of derivations, which would exactly correspond here to
the homogeneous components $\mathbbm{B}_{\eta}$ of a vector field. In the case
of the homogeneous components $\mathbbm{D}_{\eta}$ of a diffeomorphism, this
corresponds to the construction of Grossman and Larson \emph{for the vector fields}:
\[ \mathbbm{E}_{\eta} = \sum_{i_{} = 1}^{\nu} (\mathbbm{D}_{\eta} .x_{i_{}})
   \partial_{x_i} \qquad  \]

This means that the construction of the morphism $\rho^<$ only depends on the
operators
\[ \rho^< (\bullet_{\eta}) = \sum_{i_{} = 1}^{\nu} (\rho (\eta) .x_{i_{}})
   \partial_{x_i}  \quad (\text{here, the bullet designates a one vertex tree}) \]
but the origin of $\rho$ (Vector field or diffeomorphism) reappears in the
relations between $\rho$ and $\rho^<$ :
\begin{itemizeminus}
  \item In the shuffle case (Vector fields), we have, for $\eta_1 \in H$,
  \[ \rho ((\eta_1)) =\mathbbm{B}_{\eta_1} = \rho^< (\bullet^{\eta_1}) \]
  \item In the Quasishuffle case (Diffeomorphisms), if $f \in \tmmathbf{G}$ is
  given by
  \[ f_i (x) = x_i \left( 1 + \sum_{\eta \in \Eta} a^i_{\eta} x^{\eta} \right)
  \]
  then
  \[ \Theta_f = \tmop{Id}_{\mathbbm{C} [[x]]} + \sum_{s \geq 1}
     \sum_{\tmscript{\begin{array}{l}
       (\eta_1, \ldots, \eta_s) \in \Eta^s\\
       1 \leqslant i_1, \ldots, i_s \leqslant \nu
     \end{array}}} \frac{1}{s!} a^{i_1}_{\eta_1} \ldots a^{i_s}_{\eta_s}
     x^{\eta_1 + \ldots + \eta_s} x_{i_1} \ldots x_{i_x} \partial_{x_{i_1}}
     \ldots \partial_{x_{i_s}} \]
  and
  \[ \rho ((\eta)) = \sum_{\tmscript{\begin{array}{l}
       (\eta_1, \ldots, \eta_s) \in \Eta^s\\
       \eta_1 + \ldots + \eta_s = \eta
     \end{array}}} \sum_{1 \leqslant i_1, \ldots, i_s \leqslant \nu}
     \frac{1}{s!} a^{i_1}_{\eta_1} \ldots a^{i_s}_{\eta_s} x^{\eta_1 + \ldots
     + \eta_s} x_{i_1} \ldots x_{i_x} \partial_{x_{i_1}} \ldots
     \partial_{x_{i_s}} \]
  but for $\eta \in \bar{H}$, one easily sees that
  \[ \begin{array}{ccc}
       \rho ((\eta)) & = & \ds \sum_{\tmscript{\begin{array}{c}
         F = \bullet^{\eta_1} \ldots \bullet^{\eta_s}\\
         \|F\| = \eta\\
         \eta_i \in H
       \end{array}}} \rho^< (\bullet^{\eta_1} \ldots \bullet^{\eta_s})
     \end{array} \]
\end{itemizeminus}

As in section \ref{s:G}, we have

\begin{theorem}
  the map
  \[ \begin{array}{ccccc}
       S_{\rho^<} & : & \mathcal{C} (\tmop{CK}_H, \mathbbm{C}) & \rightarrow &
       \tmmathbf{G}\\
       &  & \chi & \mapsto & \tmop{ev} \left( \sum_F \chi (F) \rho^< (F)
       \right)
     \end{array} \]
  defines an antimorphism of groups and $\Theta^{\chi} = \sum_F \chi (F)
  \rho^< (F)$ is the substitution automorphism associated to $S_{\rho}
  (\chi)$.
\end{theorem}

This is a presentation of Ecalle's arborification/coarborification apparatus,
within a framework of Hopf algebras.

As we will see below, these series have many advantages in
linearization problems :
\begin{itemizeminus}
  \item Modulo a restriction to a subalgebra of $\tmop{CK}_H$ (and of its
  graded dual), strong assumptions on the spectrum will become unnecessary.
  
  \item There is a very simple criterion on characters $\chi$ in $\mathcal{C}
  (\tmop{CK}_H, \mathbbm{C})$ that ensures the analyticity of $S_{\rho^<}
  (\chi)$.
\end{itemizeminus}
Moreover, the previous computations of characters on $\tmop{Sh}_{H}$ or
$\tmop{Qsh}_H^{}$ were not useless : in many cases their computation is
easier, thanks to the simplicity of the convolution product, and for example,
once such a character $\chi$ on $\tmop{Sh}_{H}$ or $\tmop{Qsh}_H$ is
given in closed--form expression, one can easily derive a closed--form form
expression for the character $\chi^<$ on $\tmop{CK}_H$ such that
\[ S_{\rho^<} (\chi^<) = S_{\rho} (\chi) \]

 \subsection{Arborification.}

For the deconcatenation coproduct on $\tmop{Sh}_{H}$ or
$\tmop{Qsh}_H$, if $L_+^{\eta} (\tmmathbf{\eta}) = \eta \tmmathbf{\eta}$ then
\[ \Delta \circ L_+^{\eta} = 1 \otimes L_+^{\eta} + (L_+^{\eta} \otimes
   \tmop{Id}) \circ \Delta \]
we thus have the cocycle property, and then, the morphism $\alpha$ such that

\[ \alpha \circ B^+_{\eta} = L_+^{\eta} \circ \alpha \]
is a coalgebra {\tmstrong{antimorphism}} from $\tmop{CK}_H$ to $\tmop{Sh}_{H}$ or
$\tmop{Qsh}_H$ (\cite{foissy1}).

{\color{black} It is the fact that $\tmop{CK}_H$ is an initial object in a
category of coalgebras, for a certain cohomology (dual to Hochschild cohomology
of algebras) that ensures the existence of the morphism $\alpha$, which is a
morphism of Hopf algebras. We shall not expand on this (as shown by Foissy,
the cohomology groups vanish in $\tmop{degree} \geqslant 2$), yet it is
satisfactory to have such a simple algebraic characterization of
arborification through a universal property of Connes-Kreimer's algebra, which
is an important object in its own right.

  We shall now see how to recover the same diffeomorphism using $\alpha$ : going back
  to our conjugacy equations, the change of coordinates, in both cases, is
  given by a substitution automorphism
  \[ \Theta = \sum_{\tmmathbf{\eta}} \chi (\tmmathbf{\eta}) \rho
     (\tmmathbf{\eta}) \]
  and to any such character $\chi$ we have associated an arborified character
  $\chi^< = \chi \circ \alpha$. We should try to use this new character on
  $\tmop{CK}_H$ to rearrange the above series and finally get some analyticity
  properties. To do so, let us use the new Hopf algebra morphism $\rho^<$ from
  $\tmop{CK}_H^{\circ}$ to $\mathbbm{C} [x, \partial_x]$:
  \[ \Theta = \sum_{\tmmathbf{\eta}} \chi (\tmmathbf{\eta}) \rho (\tmmathbf{\eta})
     = \sum_F \chi^< (F) \rho^< (F) \]
  But then
  \[ \begin{array}{ccc}
       \ds \sum_F \chi^< (F) \rho^< (F) & = & \ds  \sum_F \chi^{} (\alpha (F)) \rho^<
       (F)\\
       & = & \ds \sum_F \chi^{} \left( \sum_{\tmmathbf{\eta}} \langle
       \tmmathbf{\eta}, \alpha (F) \rangle \tmmathbf{\eta} \right) \rho^<
       (F)\\
       & = & \ds \sum_{\tmmathbf{\eta}, F} \chi^{} (\tmmathbf{\eta}) \langle
       \tmmathbf{\eta}, \alpha (F) \rangle \rho^< (F)\\
       & = & \ds \sum_{\tmmathbf{\eta}} \chi (\tmmathbf{\eta}) \sum_F \langle
       \tmmathbf{\eta}, \alpha (F) \rangle \rho^< (F)\\
       & = & \ds \sum_{\tmmathbf{\eta}} \chi (\tmmathbf{\eta}) \sum_F \langle
       \alpha^{\circ} (\tmmathbf{\eta}), F \rangle \rho^< (F)\\
       & = & \ds \sum_{\tmmathbf{\eta}} \chi (\tmmathbf{\eta}) \rho^<
       (\alpha^{\circ} (\tmmathbf{\eta}))
     \end{array} \]
  so it appears indeed highly desirable to have such morphisms as $\rho^<$ that fulfills the relation
  \[ \rho^< \circ \alpha^{\circ} = \rho \]
  The choice of $\rho^<$ is not unique but the map defined in section 5.2
  works and it is that particular choice which has been called {\cite{snag}}
  the natural (or homogeneous) coarborification and which is adapted to the
  analytic study of $F$.

{\color{black} \begin{theorem}
  We have
  \[ \rho^< \circ \alpha^{\circ} = \rho \]
\end{theorem}
}
\begin{proof}
  $\rho$ and $\alpha^{\circ}$ are coalgebra morphisms and algebra
  antimorphisms and $\rho^<$ is a Hopf morphism, so $\rho^< \circ \alpha^{\circ}$
  and $\rho$ are coalgebra morphisms and algebra antimorphisms.
  
  In the shuffle case (Vector fields), since $\tmop{Sh}^{\circ}_{\Eta}$ is
  freely generated by the words of length 1, it is sufficient to check that
  both morphisms coincides on these words. But $\alpha^{\circ} ((\eta_1)) =
  \bullet^{\eta_1}$ thus
  \[ \rho ((\eta_1)) =\mathbbm{B}_{\eta_1} = \rho^< (\bullet^{\eta_1}) =
     \rho^< \circ \alpha^{\circ} ((\eta_1)) \]
  The same proof holds in the quasishuffle case : if $f \in \tmmathbf{G}$ is
  given by
  \[ f_i (x) = x_i \left( 1 + \sum_{\eta \in \Eta} a^i_{\eta} x^{\eta} \right)
  \]
  then
  \[ \Theta_f = \tmop{Id}_{\mathbbm{C} [[x]]} + \sum_{s \geq 1}
     \sum_{\tmscript{\begin{array}{l}
       (\eta_1, \ldots, \eta_s) \in \Eta^s\\
       1 \leqslant i_1, \ldots, i_s \leqslant \nu
     \end{array}}} \frac{1}{s!} a^{i_1}_{\eta_1} \ldots a^{i_s}_{\eta_s}
     x^{\eta_1 + \ldots + \eta_s} x_{i_1} \ldots x_{i_x} \partial_{x_{i_1}}
     \ldots \partial_{x_{i_s}} \]
  and
  \[ \rho ((\eta)) = \sum_{\tmscript{\begin{array}{l}
       (\eta_1, \ldots, \eta_s) \in \Eta^s\\
       \eta_1 + \ldots + \eta_s = \eta
     \end{array}}} \sum_{1 \leqslant i_1, \ldots, i_s \leqslant \nu}
     \frac{1}{s!} a^{i_1}_{\eta_1} \ldots a^{i_s}_{\eta_s} x^{\eta_1 + \ldots
     + \eta_s} x_{i_1} \ldots x_{i_x} \partial_{x_{i_1}} \ldots
     \partial_{x_{i_s}} \]
  but for $\eta \in \bar{H}$, one easily sees that
  \[ \begin{array}{ccc}
       \rho ((\eta)) & = & \ds \sum_{\tmscript{\begin{array}{c}
         F = \bullet^{\eta_1} \ldots \bullet^{\eta_s}\\
         \|F\| = \eta\\
         \eta_i \in H
       \end{array}}} \rho^< (\bullet^{\eta_1} \ldots \bullet^{\eta_s})\\
       & = & \rho^< (\alpha^{\circ} ((\eta))
     \end{array} \]
  and this terminates the proof.

\end{proof}}

\begin{remark}
  The mechanism of arborification of moulds has in effect been independently
  rediscovered by Ander Murua in {\cite{murua}}, involving Connes-Kreimer Hopf
  algebra, for efficient calculations involving Lie series in problems of
  control theory; in that paper, the author is then also lead to
  coarborification by considering the graded duals, and going thus to the
  Grossman--Larson algebra.
  
  In the reverse direction, Wenhua Zhao (see {\cite{zhao1}}, {\cite{zhao2}},
  {\cite{zhao3}}) has for his part rediscovered the constructions of
  coarborification and then obtained in effect the mechanisms of
  arborification by dualizing and going to $\tmop{CK}$. Notably, Zhao's
  results concern in fact both plain and contracting arborification.
  
  More recently, the universal property of $\tmop{CK}$ has also been used (in
  the non decorated case) in the same way as in our presentation, for a
  factorisation of characters of the quasishuffle algebra in \cite{cal-ef-man}.
  
  It must be stressed, however, that the crucial properties for the analyst
  come {\tmem{after{\tmem{}}}} these general constructions: namely the
  existence of closed--form expressions for the aborified moulds, which make
  it possible to obtain the necessary estimates, as we shall see below.
  
  A very striking instance, though, where an independant approach has exactly
  lead to arborification, once translated in terms of characters of the
  relevant Hopf algebras, and includes for the applications a crucial
  closed--form is \cite{foissy_unter}. Finally, in several very recent works in
  the algebraic theory of non--linear control (see \cite{lund-mk} and the
  references therein) some particular characters of the same class of Hopf
  algebras we are involved with in the present work show up, which translate
  into moulds of constant use in Ecalle's papers.
\end{remark}

\subsection{Some examples.}

\subsubsection{The shuffle case.}

For the tree
\[ t = \BeunedeuxSetroisequatre
\]
we get
\[ \alpha (t) = (\eta_1 \eta_2 \eta_3 \eta_4) + (\eta_1 \eta_3 \eta_2 \eta_4)
   + (\eta_1 \eta_3 \eta_4 \eta_2) \]
under the strong assumption on the spectrum (the ${\lambda}_i$ are independant over the integers), for
the character $\xi$ given by
\[ \xi (\eta_1, \ldots, \eta_s) = \frac{1}{\langle \lambda, \eta_1 + \ldots +
   \eta_s \rangle \langle \lambda, \eta_2 + \ldots + \eta_s \rangle \ldots
   \langle \lambda, \eta_s \rangle} \]
A simple computation yields :
\[ \xi^< (t) = \xi (\alpha (t)) = \frac{1}{\langle \lambda, \eta_1 + \eta_2 +
   \eta_3 + \eta_4 \rangle \langle \lambda, \eta_2 \rangle \langle \lambda,
   \eta_3 + \eta_4 \rangle \langle \lambda, \eta_4 \rangle} \]
For this character, even if the evaluation of $\xi^<$ on a tree involves
evaluation of $\xi$ on many sequences, there exists finally a surprisingly
simple formula for $\xi^<$ :

\begin{proposition} \label{prop:linarbo}
  Let $f$ be a tree with $s$ vertices decorated by $\eta_1, \ldots, \eta_s$.
  For $1 \leqslant i \leqslant s$ if $t_i$ is the subtree of $f$ whose root is
  labelled by $\eta_i$, then
  \[ \xi^< (f) = \prod_{i = 1}^s \frac{1}{\langle \lambda, \|t_i \| \rangle}
  \]
\end{proposition}

The reader can check this formula on the previous example where
\[ t_1 = \Ceunedeuxetrois ,\quad t_2 = \edeux ,\quad t_3 = \Setroisequatre,\quad t_4 = \equatre \]
\begin{proof}
  This result can be proved recursively on the number $s$ of vertices (i.e.
  the size of the forest). For forests of size $1$, this formula is obvious.
  
  If $f$ is a forest of size $s \geqslant 2$ with at least two trees : $f =
  t_1 \ldots t_n$ ($n \geqslant 2$), then
  \[ \xi^< (f) = \xi^< (t_1) \ldots \xi^< (t_n) \]
  but the size of each tree is less than $s$ and we get by recursion the right
  formula.
  
  If $t$ is a tree \ of size $s \geqslant 2$, then $t = B^+_{\eta} (f)$ and
  \[ \xi^< (t) = \xi (\alpha (B^+_{\eta} (f)) = \xi (L_+^{\eta} (\alpha (f)))
  \]
  but, for any sequence $\tmmathbf{\eta}$,
  \[ \xi (L_+^{\eta} (\tmmathbf{\eta})) = \frac{1}{\langle \lambda, \eta
     +\|\tmmathbf{\eta}\| \rangle} \xi (\tmmathbf{\eta}) \]
  thus
  \[ \xi^< (t) = \xi (L_+^{\eta} (\alpha (f))) = \frac{1}{\langle \lambda,
     \|t\| \rangle} \xi (\alpha (f)) = \frac{1}{\langle \lambda, \|t\|
     \rangle} \xi^< (f) \]
  and, once again, we get recursively the right formula.
\end{proof}

\subsubsection{The quasishuffle case.}

For the tree
\[ t = \BeunedeuxSetroisequatre
\]
we get
\[ \alpha (t) = (\eta_1, \eta_2, \eta_3, \eta_4) + (\eta_1, \eta_3, \eta_2,
   \eta_4) + (\eta_1, \eta_3, \eta_4, \eta_2) + (\eta_{1,} \eta_2 + \eta_3,
   \eta_4) + (\eta_1, \eta_3, \eta_2 + \eta_4) \]
under the strong assumption on the spectrum, for
the character $\chi$ given by
\[ \chi (\eta_1, \ldots, \eta_s) = \frac{1}{(l^{\eta_1 + \ldots + \eta_s} -
   1) (l^{\eta_2 + \ldots + \eta_s} - 1) \ldots (l^{\eta_s} - 1)} \]
A simple computation yields :
\[ \chi^< (t) = \chi (\alpha (t)) = \frac{1}{(l^{\eta_1 + \eta_2 + \eta_3 +
   \eta_4} - 1) (l^{\eta_2} - 1) (l^{\eta_3 + \eta_4} - 1) (l^{\eta_4} - 1)}
\]
and the same proof as before gives

\begin{proposition}
  Let $f$ be a tree with $s$ vertices decorated by $\eta_1, \ldots, \eta_s$.
  For $1 \leqslant i \leqslant s$ if $t_i$ is the subtree of $f$ whose root is
  labelled by $\eta_i$, then
  \[ \chi^< (f) = \prod_{i = 1}^s \frac{1}{(l^{\|t_i \|} - 1)} \]
\end{proposition}

Once again the formula is surprisingly simple and as we shall see in the
following section, if we have ``geometric'' estimates on such an arborified
character we will prove the analyticity of the associated diffeomorphism.

But we still have to work with strong assumptions on the spectrum. We will
circumvent this difficulty using the following remarks :
\begin{enumeratenumeric}
  \item One can obtain the above formula without arborification by translating
  directly the linearization equations as character equations on
  $\tmop{CK}_H$.
  
  \item We will then prove that, in order to define the corresponding
  diffeomorphism, it is sufficient to compute a character on a sub--Hopf
  algebra of $\tmop{CK}_H$ were the sought character is well-defined under the
  {\tmem{weak}} assumption on the spectrum. \ 
\end{enumeratenumeric}
\section{Back to linearization}\label{s:dioph}

\subsection{Equations for characters of $\tmop{CK}_H$.}

As in section \ref{s:G}, if
\[ X = X^{\tmop{lin}} + \sum_{\eta \in H} \mathbbm{B}_{\eta} = X^{\tmop{lin}}
   + P \]
with $X^{\tmop{lin}} = \sum_{1 \leqslant i \leqslant \nu} \lambda_i x_i
\partial_{x_i}$, the vector field $P$ is given by the infinitesimal character
$u$ on $\tmop{CK}_H$ :
\[ u (f) = \left\{ \begin{array}{ll}
     1 & \tmop{if} \hspace{1em} f = \bullet_{\eta}\\
     0 & \tmop{otherwise}
   \end{array} \right. \]
That is to say :
\[ X = X^{\tmop{lin}} + \sum_{\eta \in H} \mathbbm{B}_{\eta} = X^{\tmop{lin}}
   + \sum_{f \in \mathcal{F}_H} u (f) \rho^< (f) \]
The diffeomorphism $\varphi$ that linearizes $X$ ($X^{\tmop{lin}} .F_{\varphi}
= F_{\varphi} .X^{}$) can be obtained as $\varphi = S_{\rho^<} (\xi)$ where
$\xi$ is a character on $\tmop{CK}_H$ such that
\[ \nabla \xi = \xi \ast u \]
where
\[ (\nabla \chi) (f) = \langle \lambda, \|f\| \rangle \chi (f) \]
It is then easy to check directly on this equation that if $\langle \lambda,
\eta \rangle \neq 0$ for any $\eta$ in $\bar{H}$, this character is uniquely
defined and is given by proposition {\bf 8}.

On the same way, for a diffeomorphism $f^{\tmop{lin}} \circ f$ where
\[ F_f = \tmop{Id} + \sum_{\eta \in \bar{H}} \mathbbm{D}_{\eta} \]
the character $\xi$ on $\tmop{CK}_H$ given by
\[ \xi (f) = \left\{ \begin{array}{ll}
     1 & \tmop{if} \hspace{1em} f = \bullet_{\eta_1} \ldots \bullet_{\eta_s}\\
     0 & \tmop{otherwise}
   \end{array} \right. \]
is such that
\[ F_f = \sum_{f \in \mathcal{F}_H} \xi (f) \rho^< (f) \]
and if $\chi$ is a character such that
\[ \chi \circ \sigma = \chi \ast \xi \hspace{1em} (\sigma (f) = l^{\|f\|} f)
\]
then $\varphi = S_{\rho^<} (\chi)$ is such that
\[ f^{\tmop{lin}} \circ f \circ \varphi = \varphi \circ f^{\tmop{lin}} \]
Once again, if, for any $\eta \in \bar{H}$, $l^{\eta} \neq 1$, then $\chi$
is well-defined and is given by proposition {\bf 9}.

We still have the strong condition because, in order to compute such
characters on a forest $f$, one has to divide by \ $\langle \lambda, \|f\|
\rangle$ or $l^{\|f\|} - 1$ and $\|f\|$ runs over $\bar{H}$. But, as we shall
see, when considering a substitution automorphism
\[ F = \sum_{f \in \mathcal{F}_H} \chi (f) \rho^< (f) \]
there are many forests $f$ such that $\rho^< (f) = 0$. Omitting these terms in
the series defining $F$, one can consider that $f$ runs over a subset
$\mathcal{F}^+_H$ of $\mathcal{F}_H$ which is the linear basis of a sub-Hopf
algebra $\tmop{CK}^+_H$ of $\tmop{CK}^{}_H$. We will thus be able to consider
the previous character equations on $\tmop{CK}^+_H$ and there will exist a
unique solution as soon as $\langle \lambda, \eta \rangle \neq 0$ or
$l^{\eta} - 1 \neq 0$ for all $\eta$ in $H$.

\subsection{The non-resonance condition and the subalgebras of $\tmop{CK}_H$.}

\begin{definition}
  Let $\tmop{CK}^+_H$ be the subspace of $\tmop{CK}_H$ whose algebraic basis
  is given by the trees $T$ such that for any admissible cut $c$ of T where
  $(R^c (T), P^c (T)) = T_1, \ldots, T_s$, $\|T_i \|$ is in $H$ ($1 \leqslant
  i \leqslant s$). We note this set of trees $\mathcal{T}^+_H$ and the set of
  forests of such trees $\mathcal{F}^+_H$.
\end{definition}

It is readily checked that $\tmop{CK}^+_H$ is a sub--Hopf algebra of
$\tmop{CK}_H$. But one can also prove the following :

\begin{theorem}
  \label{th:restrict}If a forest F in $\mathcal{F}_H$ does not belong to
  $\mathcal{F}^+_H$, then
  \[ \rho^< (F) = 0 \]
\end{theorem}

\begin{proof}
  Starting with with a diffeomorphism or a vector field, it is clear that for
  $\eta \in H$, the image of the one node tree, decorated by $\eta$, we have:
  \[ \rho^< (\eta) = \sum_{i = 1}^{\nu} u^i_{\eta} x^{\eta + e_i}
     \partial_{x_i} \]
  where $u^i_{\eta} \in \mathbbm{C}$. For any forest $F = T_1 \ldots T_s$ in \
  $\mathcal{F}_H$, $\rho^< (F)$ is an endomorphism of $\mathbbm{C} [x]$ such
  that
  \[ \rho^< (T_1 \ldots T_s) = \sum_{1 \leqslant i_1, \ldots, i_s \leqslant
     \nu} P_F^{i_1, \ldots i_s} (u) x^{\|F\| + e_{i_1} + \ldots + e_{i_s}}
     \partial_{x_{i_1}} \ldots \partial_{x_{i_s}} \]
  where the coefficients $P_F^{i_1, \ldots i_s} (u)$ are polynomials in the
  variables $u = \{u^i_{\eta} \}$ with coefficients in $\mathbbm{Q}^+$
  ($P_F^{i_1, \ldots i_s} (u) \in \mathbbm{Q}^+ [u]$).
  
  Let us first consider a tree $T$ in \ $\mathcal{T}_H$ such that $\|T\|
  \not\in H$. This means that, for $1 \leqslant i \leqslant \nu$, $\|T\| +
  e_i \not\in \mathbbm{N}^{\nu}$. But
  \[ \rho^< (T) = \sum_{i = 1}^{\nu} P_T^{i_{}} (u) x^{\|T\| + e_{i_{}}}
     \partial_{x_{i_{}}} \]
  with
  \[ P_T^{i_{}} (u) x^{\|T\| + e_{i_{}}} = \rho^< (T) .x_i \in \mathbbm{C}
     [x] \] 
  and, since $x^{\|T\| + e_{i_{}}}$ is not in $\mathbbm{C} [x]$, for $1
  \leqslant i \leqslant \nu$, $P_T^{i_{}} (u) = 0$ and $\rho^< (T) = 0$. Now,
  from the recursive definition of $\rho^<$, if $F = T_1 \ldots T_s$ with at
  least one tree $T_{i_0}$ such that $\|T_{i_0} \| \not\in H$, then
  \[ \rho^< (T_1 \ldots T_s) = 0 \]
  Now let $T$ be a tree such that there exists an admissible cut $c$ of T
  where $(R^c (T), P^c (T)) = (T_0, T_1 \ldots T_s)$ with at least one $\|T_i
  \| \not\in H$ ($T \in \mathcal{T}_H /\mathcal{T}^+_H$). From the previous
  property one can deduce that either $\rho^< (T_0) = 0$ or $\rho^< (T_1
  \ldots T_s) = 0$ thus,
  \[ \rho^< (T_1 \ldots T_s) . \rho^< (T_0) = 0 = \rho^< \circ \pi (P^c (T)
     \otimes R^c (T)) \]
  where $\pi$ is the the product in $\tmop{CK}^{\circ}_H$, dual to the
  coproduct of $\tmop{CK}_H$. But thanks to the definition of this coproduct
  \[ \pi (P^c (T) \otimes R^c (T)) = c_{} T + Q \]
  where $c \in \mathbbm{N}^{\ast}$ and $Q$ is a combination of forests with
  coefficients in $\mathbbm{N}$. Now
  \[ \rho^< (T_1 \ldots T_s) . \rho^< (T_0) .x_i = 0 = c P_T^{i_{}} (u)
     x^{\|T\| + e_i} + \rho^< (Q) .x_i = (c P_T^{i_{}} (u) + Q^i (u)) x^{\|T\|
     + e_i} \]
  this means that the polynomial $c P_T^{i_{}} (u) + Q^i (u)$ is zero but,
  since it is a linear combination (with positive coefficients) of polynomials
  in $\mathbbm{Q}^+ [[u]]$,
  \[ P_T^{i_{}} (u) = Q^i (u) = 0 \]
  and then $\rho^< (T) = 0$. Using once again the recursive definition of
  $\rho^<$, we obtain that if $F \in \mathcal{F}_H /\mathcal{F}_H^+$, $\rho^<
  (F) = 0$.
\end{proof}

This means that, if $p$ is the projection \ of $\tmop{CK}_H$ on
$\tmop{CK}^+_H$ or of $\tmop{CK}^{\circ}_H$ on $\tmop{CK}^{+ \circ}_H$ defined
by
\[ \forall F \in \mathcal{F}_H, \hspace{1em} p (F) = \left\{
   \begin{array}{lll}
     0 & \tmop{if} & F \not\in \mathcal{F}^+_H\\
     F & \tmop{if} & F \in \mathcal{F}^+_H
   \end{array} \right. \]
then $\rho^< \circ p$ is still a Hopf morphism. Moreover, for any character on
$\chi$ on \ $\tmop{CK}_H$ or $\tmop{CK}_{\bar{H}}$, $\chi \circ p$ is a
character on $\tmop{CK}^+_H$ and
\[ \sum_F \chi (F) \rho^< (F) = \sum_F \chi (p (F)) \rho^< (p (F)) = \sum_{F
   \in \tmop{CK}^+_H} \chi (F) \rho^< (F) \]
In other words, in the linearization equation, one can look for a substitution
morphism given by a character on $\tmop{CK}^+_H$ and this one is well-defined
as soon as we have the {\tmem{weak}} non--resonance condition.

Thus, this Hopf algebra $\tmop{CK}^+_H$, which does not appear in the
literature, is the relevant object one has to use, in order to recover the
usual results on formal linearization :
\begin{enumeratenumeric}
  \item It works with the classical conditions on the spectrum; no extra
  assumption is needed.
  
  \item The diffeomorphism is expressed by a character which is given without
  ambiguity.
\end{enumeratenumeric}
It remains to prove that $\tmop{CK}^+_H$ is also extremely well-suited to
consider the analyticity of such a diffeomorphism. In other words, the Hopf
algebra $\tmop{CK}^+_H$ is the right algebra to deal with questions of
convergence in linearization problems (and in fact also in more general
normalization problems, in situations involving resonances).

\subsection{Majorant series and analyticity.}

Using majorant series, it is easy to see that
\[ \tmmathbf{G}_{\tmop{ana}} = \{\varphi = (\varphi_1, \ldots, \varphi_{\nu})
   \in \tmmathbf{G} \hspace{1em} ; \hspace{1em} \varphi_i (x) \in
   \mathbbm{C}\{x\}\} \]
is a subgroup of $\tmmathbf{G}$ and this still holds for many subsets of
diffeomorphisms whose coefficients satisfy some particular estimates (see
\cite{menous_birkh}).

\begin{theorem}
  Let $B = \{B_{\eta} \in \mathbbm{R}^+, \hspace{1em} \eta \in H\}$ be a set
  of submultiplicative estimates : for all $\eta, \eta_1, \eta_2$ in $H$ such
  that $\eta = \eta_1 + \eta_2$, $B_{\eta_1} B_{\eta_2} \leqslant B_{\eta} =
  B_{\eta_1 + \eta_2}$. Let $\tmmathbf{G}_B$ be the subset of \
  $\tmmathbf{G}_{}$ of diffeomorphisms $\varphi$ such that there exists $A >
  0$ and
  \[ \forall 1 \leqslant i \leqslant \nu, \forall \eta \in H_i, \hspace{1em} |
     \varphi^i_{\eta} | \leqslant B_{\eta} A^{| \eta |} \]
  Then $\tmmathbf{G}_B$ is a subgroup of $\tmmathbf{G}$.
\end{theorem}

The complete proof can be found in \cite{menous_birkh}. It relies on majorant series:
let $\varphi (x) = x + u (x)$ in $\tmmathbf{G}$, we say that $\psi (x) = x + v
(x)$ is a majorant series of $\varphi$ ($\varphi \prec \psi$) if,
\[ \forall 1 \leqslant i \leqslant \nu, \forall \eta \in H_i, \hspace{1em} |
   \varphi^i_{\eta} | \leqslant \psi^i_{\eta} \]
For a given set $B$ and $A > 0$, let $\psi_{B, A}$ be the diffeomorphism such
that $C^i_{\eta} (\psi_{B, A}) = B_{\eta} A^{| \eta |}$. It is clear that
$\psi_{B, A}$ is in $\tmmathbf{G}_B$ and $\varphi$ belongs to $\tmmathbf{G}_B$
if and only if there exists $A > 0$ such that
\[ \varphi \prec \psi_{B, A} \]
Now the proof of the theorem relies on classic estimates that gives:
\begin{enumeratenumeric}
  \item If $\varphi_1 \prec \psi_{B, A_1}$ and $\varphi_2 \prec \psi_{B, A_2}$
  then there exists \ $A_3 > 0$ such that $\varphi_1 \circ \varphi_2 \prec
  \psi_{B, A_3}$. In other words, $\tmmathbf{G}_B$ is stable under the
  composition of diffeomorphisms.
  
  \item $\tmop{If} \varphi_1 \prec \psi_{B, A_1}$ then there exists \ $A_2 >
  0$ such that $\varphi_1^{\circ^{- 1}} \prec \psi_{B, A_2}$ and this finally
  proves that $\tmmathbf{G}_B$ is a subgroup.
\end{enumeratenumeric}
Note that the analytic subgroup corresponds to $\tmmathbf{G}_B$ with,
\[ \forall \eta \in H, \hspace{1em} B_{\eta} = 1 \]
Now, using the same ideas as in {\cite{menous_birkh}}, one easily gets that

\begin{theorem}
  Suppose that, the map $\rho^<$, restricted to $\tmop{CK}^{+ \circ}_H$ is
  such that :
  \[ \rho^< (\bullet_{\eta}) = \sum_{1 \leqslant i \leqslant \nu} u^i_{\eta}
     x^{\eta} x_i \partial_{x_i} \]
  with $|u_{\eta}^i | \leqslant B_{\eta} A^{| \eta |}$ for some $A > 0$. If
  $\chi$ is a character on $\tmop{CK}^+_H$ such that, for all forests $F \in
  \tmop{CK}^+_H$,
  \[ | \chi (F) | \leqslant C^{\tmop{gr} (F)} \]
  then the diffeomorphism $\varphi$ such that
  \[ \Theta_{\varphi} = \sum_{F \in \tmop{CK}^+_H} \chi (F) \rho^< (F) \]
  is in $\tmmathbf{G}_B$.
\end{theorem}

\begin{proof}
  If we consider
  \[ u (x) = (u_1 (x), \ldots, u_{\nu} (x)) \]
  where
  \[ u_i (x) = x_i + \sum_{\eta \in H} \rho^< (\bullet_{\eta}) .x_i = x_i +
     \sum_{\eta \in H_i} u^i_{\eta} x^{\eta} x_i \]
  then $u \prec \psi_{B, A} = v$. We note $\rho_u = \rho^<$ and $\rho_v$ the
  similar morphism such that
  \[ \rho_v (\bullet_{\eta}) = \sum_{1 \leqslant i \leqslant \nu} v^i_{\eta}
     x^{\eta} x_i \partial_{x_i} \]
  For any forest $F = T_1 \ldots T_s$ in $\mathcal{F}^+_H$, we have once \
  again
  \[ \rho^{}_u (T_1 \ldots T_s) = \sum_{1 \leqslant i_1, \ldots, i_s \leqslant
     \nu} P_F^{i_1, \ldots i_s} (u) x^{\|F\| + e_{i_1} + \ldots + e_{i_s}}
     \partial_{x_{i_1}} \ldots \partial_{x_{i_s}} \]
  where the coefficients $P_F^{i_1, \ldots i_s} (u)$ are polynomials in the
  variables $u = \{u^i_{\eta} \}$ with coefficients in $\mathbbm{Q}^+$
  ($P_F^{i_1, \ldots i_s} (u) \in \mathbbm{Q}^+ [u]$). Since the coefficients
  of such polynomials are non-negative, it is clear that
  \[ |P_F^{i_1, \ldots i_s} (u) | \leqslant P_F^{i_1, \ldots i_s} (v) \]
  and if
  \[ \varphi (x) = \Theta_{\varphi} .x = \sum_{F \in \tmop{CK}^+_H} \chi (F)
     \rho_u (F) .x \]
  we have that
  \[ \varphi (x) \prec \sum_{F \in \tmop{CK}^+_H} | \chi (F) | \rho_v (F) .x
     \prec \sum_{F \in \tmop{CK}^+_H} C^{\tmop{gr} (F)} \rho_v (F) .x = \phi
     (x) \in \tmmathbf{G} \]
  The map $\xi$ defined on $\tmop{CK}^+_H$ by $\xi (F) = C^{\tmop{gr} (F)}$ is
  a character and it is easy to check that its inverse is defined by
  \[ \xi^{\ast^{- 1}} (F) = \left\{ \begin{array}{ll}
       1 & \tmop{if} \hspace{1em} F = \emptyset\\
       (- 1)^s C^{\tmop{gr} (F)} & \tmop{if} \hspace{1em} F = \bullet_{\eta_1}
       \ldots \bullet_{\eta_s}\\
       0 & \tmop{otherwise}
     \end{array} \right. \]
  This means that
  \[ \begin{array}{ccc}
       \phi^{\circ^{- 1}} (x) & = & \ds \sum_{F \in \tmop{CK}^+_H} \xi^{\ast^{-
       1}} (F) \rho_u (F) .x\\
       & = & x +\ds \sum_{\eta \in H} \xi^{\ast^{- 1}} (\bullet_{\eta}) \rho_v
       (\bullet_{\eta}) .x\\
       & = & x - \ds \sum_{\eta \in H} C^{\tmop{gr} (\eta)} \rho_v
       (\bullet_{\eta}) .x\\
       & = & 2 x - \ds \left( x + \sum_{\eta \in H} C^{\tmop{gr} (\eta)} \rho_v
       (\bullet_{\eta}) .x \right)\\
       & = & 2 x - \frac{1}{C} v (C x)
     \end{array} \]
  As $v$ is in $\tmmathbf{G}_B$, so is $\phi^{\circ^{- 1}}$ and, since
  $\tmmathbf{G}_B$ is a group, we \ have
  \[ \varphi (x) \prec \phi (x) \in \tmmathbf{G}_B \]
  and $\varphi$ is in $\tmmathbf{G}_B$.
\end{proof}

The previous argument is a systematization of a process that was introduced by
one of us (FM), and implemented in 2 previous papers ({\cite{menous_qdiff}},
{\cite{menous_birkh}}), regarding respectively non--linear q--difference
equations and ``Birkhoff decomposition'' in spaces of Gevrey series.

\subsection{Growth estimates for the arborified moulds}

In order to give a nontrivial application, we show how Brjuno's classical
result on linearization for non resonant fields can be obtained, once we match
the previous estimate on the comould side with another one, regarding the
geometric growth of the arborified mould.

We start with the vector field case and we denote by $M^{\bullet^<}$ the arborescent mould
corresponding to the character $\xi$, for which a closed--form expression was
obtained above.

In order to avoid technicalities in diophantine approximation, in the present
work which is focussed on algebraic constructions, we shall consider vector
fields which satisfy the following strong version of Brjuno's diophantine
condition:
\[ \sum \frac{1}{2^k} \tmop{Log} ( \frac{1}{\Omega (2^{k + 1})}) < \infty
  \]
where $\Omega ( h) = \Omega (h) = \min \left\{ | \langle n, \lambda \rangle |,
n_i \in \mathbbm{Z} \nocomma, | \langle n, \lambda \rangle | > 0 \tmop{and}
\sum n_i \leqslant h \right\}$

Note that, since $\Omega$ is decreasing, the above condition is equivalent to the condition :
\[ S=\sum \frac{1}{2^k} \left | \tmop{Log} ( \frac{1}{\Omega (2^{k + 1})}) \right | < \infty
 \]

\begin{proposition}
  The arborified mould $M^{\bullet^<} $has a geometric growth : there exists a constant
  $K$ such that for any 
decorated forest $F$, we have $|
  M^{F} | \leqslant K^{\tmop{gr} (F)}$.
\end{proposition}

All proofs of normalization results under Brjuno's arithmetical condition rely
at some point on an key estimate, usually known as a ``Brjuno's counting
lemma'' (see e.g. the classical paper by J. Poeschel {\cite{poeschel}} for
a particularly clear exposition of this, in the case of diffeomorphisms). The
proof of the previous proposition will unsurprisingly also crucially depend as well on
a version of a counting lemma which we give below. In the form that we use
here, the estimate is proved in the paper {\cite{carletti}}, for the version of
the lemma that is relevant in the case of diffeomorphisms). The paper
{\cite{gallavotti}}, for example, explains the way trees appear naturally in
this context; it is then straightforward to translate the version of the
counting lemma for fields which is contained in Brjuno's seminal paper in the
language of trees.

Note however that our presentation is different and totally independant of
the one used in {\cite{gallavotti}} and {\cite{carletti}} but it is the very
same counting argument that is crucial, as in any other proofs of results
involving Brjuno's condition. The proof of the proposition itself will simply
consist in regrouping subtrees in ``slices'' that are determined by a total weight
comprised between 2 successive values of $\Omega ( 2^l)$.

\begin{lemma}
  (Tree version of ``Brjuno's counting lemma'')
  
  Let $F$ a decorated forest with $r$ vertices  and let $s   = \tmop{gr} (F)$ (we consider here only forests such that
  $\langle \tmmathbf{\lambda}, \| F \| \rangle \neq 0$). If, for any nonnegative integer $k$, $N_k(F)$ is the number of subtrees
  $t$ of $F$ that satisfy the following
  inequality :
  \[  \frac{1}{2} \Omega (2^{k + 1}) \leqslant \langle \tmmathbf{\lambda}, \|
     t \| \rangle < \frac{1}{2} \Omega (2^k), \]
     then
  \[ N_k(F) \leqslant \left\{ \begin{array}{l}
       0 \tmop{if} s<2^k\\
       \ds 2 \nu \frac{s}{2^k} - 1 \tmop{if} 2^k \leqslant s
     \end{array} \right. \]
\end{lemma}

Let us now consider a forest $F$ with $r$ vertices decorated by $\eta_1, \ldots, \eta_r$. Let $s =
\tmop{gr} (F)$ and let $l$ be the integer such that $2^l
\leqslant s < 2^{l + 1}$. The closed--form expression of the mould (see proposition \ref{prop:linarbo}) is given by:
\[ M^{F} = \prod_{i = 1}^r \frac{1}{\langle \lambda, \|t_i \| \rangle}\]
where, for $1 \leqslant i \leqslant r$ if $t_i$ is the subtree of $F$ whose root is
  labelled by $\eta_i$.
   
We immediately obtain :
\[ | M^{F} | \leqslant \prod_{k = 0}^l \left ( \frac{2}{\Omega (2^{k +
   1})} \right )^{N_k  (F)} = 2^r\prod_{k=0}^l\left ( \frac{1}{\Omega (2^{k +
   1})} \right)^{N_k  (F)} \]
Thanks to the previous lemma, for indices $k\leqslant l$, $N_k(F)\leqslant  2 \nu \frac{s}{2^k}$ thus
\[ \begin{array}{rcl}
| M^{F} | &\leqslant & \ds 2^r \exp \left (\sum_{k=0}^l N_k(F) \tmop{Log} ( \frac{1}{\Omega (2^{k + 1})})\right ) \\
 &  \leqslant   & 2^r \ds\exp \left (\sum_{k=0}^l 2 \nu \frac{s}{2^k} \left | \tmop{Log} ( \frac{1}{\Omega (2^{k + 1})}) \right |\right ) \\
 & \leqslant & 2^r \exp(2\nu s S) \leqslant C^{\tmop{gr}(F)} 
 \end{array}\]
 with $C=2\exp(2\nu S)$.

The case of diffeomorphisms is settled in exactly the same way. We denote by
$N^{\bullet^<}$ the linearizing mould that corresponds to the character $\chi$, and we
have the following:

\begin{proposition}
  The arborified mould $N^{\bullet^<}$has a geometric growth : there exists a constant
  $D$ such that $| N^{F} | \leqslant D^{\tmop{gr}
  (F)}$
\end{proposition}

The proof goes along the same lines as for vector fields, using instead the
following closed form:
\[ N^{F} = \prod_{i = 1}^r \frac{1}{e^{2 \pi i \langle \lambda, \|t_i \| \rangle}-1}\]
where, for $1 \leqslant i \leqslant r$ if $t_i$ is the subtree of $F$ whose root is
  labelled by $\eta_i$.

and applying the relevant counting lemma as in {\cite{carletti}}.

\subsection{The analytic normalization scheme with $\tmop{CK}_H^+$}

Let us recollect now the scheme for linearizing a non resonant dynamical
system, using the Hopf algebra $\tmop{CK}_H^+$ :
\begin{enumeratenumeric}
  \item We express the equation regarding the normalizing substitution
  automorphism as an equation on characters of $\tmop{CK}_H^+$
  
  \item We solve this equation, obtaining this way a well-defined character,
  even for forests displaying ``fake resonances'' for some of their subtrees
  
  \item We prove some geometrical growth estimate for this character, by using
  the Diophantine hypothesis on the spectrum
  
  \item We match this with the geometric growth for the comould part in the
  expansion
  
  \item We obtain a convergent series of operators, which makes it possible
  to conclude to the analyticity of the transformation thus constructed
\end{enumeratenumeric}
So in fact, stricly speaking, we don't need to arborify moulds, we can work \
from the outset at the arborescent level, and directly at the level of the
algebra $\tmop{CK}^+$, which is the one the underlies all the computations,
and for which no fake obstruction remain. However, plain (i. e. non
arborescent) moulds are nevertheless very useful because it is usually easier
to guess a closed--form expression for them, before proving that their
arborescent counterparts also have a closed--form of the same kind (and this
is a very general phenomenon for the use of arborification, cf
\cite{menous_qdiff} and \cite{menous_birkh}).

To dispell any idea that the scheme we have described in the present text is
too special and only limited to giving a new proof of already well known
results achieved by common methods, let us indicate 2 directions:

-- A natural question is the linearization of nonresonant dynamical systems
for data of various classes of regularity. In {\cite{carletti}}, the author
proved new results of linearization for diffeomorphisms or vector fields which
are formal series with Gevrey growth estimates, under a Brjuno condition. It
is straightforward to get the same results with the mould apparatus, using the
approach detailed in the present text. The algebraic constructions are
{\tmem{exactly}} the same, all the results on the mould side can be used
unchanged, the only supplementary thing is to show the geometric growth for
the comould part, adapted to spaces of Gevrey series, instead of analytic
ones, which is easy. Now, the point is that in order to go beyond such results
performed on {\tmem{formal}} spaces of series with some growth conditions, to
tackle the same question for {\tmem{functional spaces}}, e. g. data which are
summable in one variable, or multisummable, or resurgent, the same scheme
remains valid in the mould/comould formalism, whereas under other approaches 
would require ad hoc estimates that would be quite difficult to prove.

-- Next we can consider the question of normalization of resonant local
dynamical systems; there, linearization is generically not possible using
formal series, but there are simple normal forms and the normalizing series
are generically divergent ({\cite{martinet-ramis}}, {\cite{snag}}). The
substitution automorphisms for the normalizing transformations can be
expressed by mould/comould expansions, where the moulds take their values in
some algebra $\mathcal{R}$ of resurgent functions {\cite{snag}}. In the
presence of diophantine small denominators, the
arborification/coarborification machinery is used in the same way as in the
present paper; in this Hopf--algebraic presentation, arborification is a
factorization of characters from $\tmop{CK}_{}$ to the (commutative) algebra
$\mathcal{R}$ and the comould constructions are exactly the same (and
$\tmop{CK}_+$ plays an important role, there, too). All the constructions and
theorems are already in Ecalle's foundational papers, but with arguments that
are very concise; the presentation we give yield easy proofs of algebraic
properties of the arborification formalism and makes it possible to connect it
to some very recent work in algebraic combinatorics. Applications of
arborescent moulds go much further than its original domain of application,
namely irregular singularities of local dynamical systems : Stochastic
Processes, in particular the theory of rough paths is one striking example
(see in particular section 4.2 of {\cite{foissy_unter}}, where the concept of
extension is {\tmem{exactly}} the factorization of characters as we have
formulated it; see also {\cite{ef_et_al}}); the fast expanding algebraic
theory of non--linear control theory, with Hopf algebraic formulations of
(Lie--)Butcher series is another one ({\cite{lund-mk}}).

{\color{black} \section{Conclusion.}\label{s:concl}

Ecalle's mould--comould formalism has been in the present text given a
presentation in terms of some Hopf algebras (Fa{\`a} di Bruno, shuffle,
quasishuffle, Connes--Kreimer, Grossman--Larson...) which are by now standard
objects in algebraic combinatorics. In this way, symmetral moulds appear as
characters of a decorated shuffle Hopf algebra, and symmetrel ones as
characters of a quasishuffle one. Next we have shown that arborification
(resp. contracting arborification) of moulds is the outcome of a factorization
of characters, by using a universal property satisfied by Connes-Kreimer Hopf
algebra.

Then, going to the graded duals, we have been able to {\tmem{characterize}}
the fundamental process of {\tmem{homogeneous coarborification}} in a simple
way, and consequently easily obtaining justifications of its properties, by
building on known facts regarding Grossman--Larson Hopf algebra.

We have introduced a {\tmem{subalgebra of the decorated Connes--Kreimer
algebra}} which underlies the calculations of normalization of analytic
dynamical systems at singularities. Namely, computing a normalizing
transformation will amount to finding a character of this algebra, which
satisfies a particular equation that directly comes from the normalization
relation itself. In the present paper, we have illustrated the method by the
well-known problem of linearization of non--resonant dynamical systems in any
dimension, in the presence of small denominators. In problems involving
resonances together with small denominators, the same Hopf--algebraic
apparatus governs the calculations and the only thing that changes is that the
characters are not scalar any more but take their values in relevant
algebras of resurgent functions.

\vspace{1cm}

Fr{\'e}d{\'e}ric Fauvet, IRMA, Universit{\'e} de Strasbourg, 7 rue Descartes

67084 \ Strasbourg Cedex, France. {\small\tt fauvet@math.unistra.fr}

\vspace{0.5cm}

Fr{\'e}d{\'e}ric Menous, B{\^a}t. 425, Universit{\'e} Paris-Sud, 91405 Orsay
Cedex, France. 

{\small\tt Frederic.Menous@math.u-psud.fr}

\end{document}